\documentclass [12pt]{amsart}
\usepackage{amscd,amssymb,latexsym}
\usepackage{tikz}
\usepackage{fullpage}
\newcommand{\Mdef}[2]{\newcommand{#1}{\relax \ifmmode #2 \else $#2$\fi}}

\newcommand{\sm }{\wedge}

\newcommand{\tensor}{\otimes}

\Mdef{\bhom}{\mathbf{\hat{H}om}}

\Mdef{\Mod}{\mathrm{mod}}

\newcommand{\st}{\; | \;}

\newtheorem{thm}{Theorem}[section]
\newtheorem{lemma}[thm]{Lemma}

\newtheorem{cor}[thm]{Corollary}

\theoremstyle{definition}

\newtheorem{defn}[thm]{Definition}

\newtheorem{remark}[thm]{Remark}

\newcommand{\qqed}{\qed \\[1ex]}
\renewenvironment{proof}[1][\hspace*{-.8ex}]{\noindent {\bf Proof #1:\;}}{\qqed}

\Mdef{\PH} {\Phi^H}
\Mdef{\PK} {\Phi^K}
\Mdef{\PL} {\Phi^L}
\Mdef{\PT} {\Phi^{\T}}

\Mdef{\ef}{E{\cF}_+}
\Mdef{\etf}{\widetilde{E}{\cF}}
\Mdef{\eg}{E{G}_+}
\Mdef{\etg}{\tilde{E}{G}}

\Mdef{\infl}{\mathrm{inf}}
\Mdef{\defl}{\mathrm{def}}
\Mdef{\res}{\mathrm{res}}
\Mdef{\ind}{\mathrm{ind}}
\Mdef{\coind}{\mathrm{coind}}

\Mdef{\univ}{\mathcal{U}}

\Mdef{\Fp}{\mathbb{F}_p}
\Mdef{\Zpinfty}{\Z /p^{\infty}}
\Mdef{\Zpadic}{\Z_p^{\wedge}}

\newcommand{\bi}{\begin{itemize}}
\newcommand{\be}{\begin{enumerate}}
\newcommand{\bc}{\begin{center}}
\newcommand{\bd}{\begin{description}}
\newcommand{\ei}{\end{itemize}}
\newcommand{\ee}{\end{enumerate}}
\newcommand{\ec}{\end{center}}
\newcommand{\ed}{\end{description}}

\newcommand{\lra}{\longrightarrow}
\newcommand{\lla}{\longleftarrow}

\Mdef{\we}{\mathbf{we}}
\Mdef{\fib}{\mathbf{fib}}
\Mdef{\cof}{\mathbf{cof}}
\Mdef{\BI}{\mathcal{BI}}

\Mdef{\B}{\mathbb{B}}
\Mdef{\C}{\mathbb{C}}
\Mdef{\D}{\mathbb{D}}
\Mdef{\E}{\mathbb{E}}
\Mdef{\T}{\mathbb{T}}
\Mdef{\F}{\mathbb{F}}
\Mdef{\G}{\mathbb{G}}
\Mdef{\I}{\mathbb{I}}
\Mdef{\N}{\mathbb{N}}
\Mdef{\Q}{\mathbb{Q}}
\Mdef{\R}{\mathbb{R}}
\Mdef{\bbS}{\mathbb{S}}
\Mdef{\Z}{\mathbb{Z}}

\Mdef{\bA}{\mathbb{A}}
\Mdef{\bB}{\mathbb{B}}
\Mdef{\bC}{\mathbb{C}}
\Mdef{\bD}{\mathbb{D}}
\Mdef{\bE}{\mathbb{E}}
\Mdef{\bF}{\mathbb{F}}
\Mdef{\bG}{\mathbb{G}}
\Mdef{\bH}{\mathbb{H}}
\Mdef{\bI}{\mathbb{I}}
\Mdef{\bJ}{\mathbb{J}}
\Mdef{\bK}{\mathbb{K}}
\Mdef{\bL}{\mathbb{L}}
\Mdef{\bM}{\mathbb{M}}
\Mdef{\bN}{\mathbb{N}}
\Mdef{\bO}{\mathbb{O}}
\Mdef{\bP}{\mathbb{P}}
\Mdef{\bQ}{\mathbb{Q}}
\Mdef{\bR}{\mathbb{R}}
\Mdef{\bS}{\mathbb{S}}
\Mdef{\bT}{\mathbb{T}}
\Mdef{\bU}{\mathbb{U}}
\Mdef{\bV}{\mathbb{V}}
\Mdef{\bW}{\mathbb{W}}
\Mdef{\bX}{\mathbb{X}}
\Mdef{\bY}{\mathbb{Y}}
\Mdef{\bZ}{\mathbb{Z}}

\Mdef{\cA}{\mathcal{A}}
\Mdef{\cB}{\mathcal{B}}
\Mdef{\cC}{\mathcal{C}}
\Mdef{\mcD}{\mathcal{D}} 
\Mdef{\cE}{\mathcal{E}}
\Mdef{\cF}{\mathcal{F}}
\Mdef{\cG}{\mathcal{G}}
\Mdef{\mcH}{\mathcal{H}} 
\Mdef{\cI}{\mathcal{I}}
\Mdef{\cJ}{\mathcal{J}}
\Mdef{\cK}{\mathcal{K}}
\Mdef{\mcL}{\mathcal{L}}

\Mdef{\cM}{\mathcal{M}}
\Mdef{\cN}{\mathcal{N}}
\Mdef{\cO}{\mathcal{O}}
\Mdef{\cP}{\mathcal{P}}
\Mdef{\cQ}{\mathcal{Q}}
\Mdef{\mcR}{\mathcal{R}}
\Mdef{\cS}{\mathcal{S}}
\Mdef{\cT}{\mathcal{T}}
\Mdef{\cU}{\mathcal{U}}
\Mdef{\cV}{\mathcal{V}}
\Mdef{\cW}{\mathcal{W}}
\Mdef{\cX}{\mathcal{X}}
\Mdef{\cY}{\mathcal{Y}}
\Mdef{\cZ}{\mathcal{Z}}

\Mdef{\ca}{\mathcal{a}}
\Mdef{\ct}{\mathcal{t}}

\Mdef{\At}{\tilde{A}}
\Mdef{\Bt}{\tilde{B}}
\Mdef{\Ct}{\tilde{C}}
\Mdef{\Et}{\tilde{E}}
\Mdef{\Ht}{\tilde{H}}
\Mdef{\Kt}{\tilde{K}}
\Mdef{\Lt}{\tilde{L}}
\Mdef{\Mt}{\tilde{M}}
\Mdef{\Nt}{\tilde{N}}
\Mdef{\Pt}{\tilde{P}}

\Mdef{\tA}{\tilde{A}}
\Mdef{\tB}{\tilde{B}}
\Mdef{\tC}{\tilde{C}}
\Mdef{\tE}{\tilde{E}}
\Mdef{\tH}{\tilde{H}}
\Mdef{\tK}{\tilde{K}}
\Mdef{\tL}{\tilde{L}}
\Mdef{\tM}{\tilde{M}}
\Mdef{\tN}{\tilde{N}}
\Mdef{\tP}{\tilde{P}}

\Mdef{\ft}{\tilde{f}}
\Mdef{\xt}{\tilde{x}}
\Mdef{\yt}{\tilde{y}}

\Mdef{\Ab}{\overline{A}}
\Mdef{\Bb}{\overline{B}}
\Mdef{\Cb}{\overline{C}}
\Mdef{\Db}{\overline{D}}
\Mdef{\Eb}{\overline{E}}
\Mdef{\Fb}{\overline{F}}
\Mdef{\Gb}{\overline{G}}
\Mdef{\Hb}{\overline{H}}
\Mdef{\Ib}{\overline{I}}
\Mdef{\Jb}{\overline{J}}
\Mdef{\Kb}{\overline{K}}
\Mdef{\Lb}{\overline{L}}
\Mdef{\Mb}{\overline{M}}
\Mdef{\Nb}{\overline{N}}
\Mdef{\Ob}{\overline{O}}
\Mdef{\Pb}{\overline{P}}
\Mdef{\Qb}{\overline{Q}}
\Mdef{\Rb}{\overline{R}}
\Mdef{\Sb}{\overline{S}}
\Mdef{\Tb}{\overline{T}}
\Mdef{\Ub}{\overline{U}}
\Mdef{\Vb}{\overline{V}}
\Mdef{\Wb}{\overline{W}}
\Mdef{\Xb}{\overline{X}}
\Mdef{\Yb}{\overline{Y}}
\Mdef{\Zb}{\overline{Z}}

\Mdef{\db}{\overline{d}}
\Mdef{\hb}{\overline{h}}
\Mdef{\qb}{\overline{q}}
\Mdef{\rb}{\overline{r}}
\Mdef{\tb}{\overline{t}}
\Mdef{\ub}{\overline{u}}
\Mdef{\vb}{\overline{v}}

\Mdef{\hc}{\hat{c}}
\Mdef{\he}{\hat{e}}
\Mdef{\hf}{\hat{f}}
\Mdef{\hA}{\hat{A}}
\Mdef{\hH}{\hat{H}}
\Mdef{\hJ}{\hat{J}}
\Mdef{\hM}{\hat{M}}
\Mdef{\hP}{\hat{P}}
\Mdef{\hQ}{\hat{Q}}

\Mdef{\thetab}{\overline{\theta}}
\Mdef{\phib}{\overline{\phi}}

\Mdef{\uA}{\underline{A}}
\Mdef{\uB}{\underline{B}}
\Mdef{\uC}{\underline{C}}
\Mdef{\uD}{\underline{D}}

\Mdef{\bolda}{\mathbf{a}}
\Mdef{\boldb}{\mathbf{b}}
\Mdef{\bfD}{\mathbf{D}}

\Mdef{\fm}{\frak{m}}
\Mdef{\fp}{\frak{p}}

\Mdef{\eps}{\epsilon}

\input{xypic}

\begin{document}
\title{Four approaches to cohomology theories with Reality}

\author{J.P.C.Greenlees}
\address{School of Mathematics and Statistics, Hicks Building, 
Sheffield S3 7RH. UK.}
\email{j.greenlees@sheffield.ac.uk}
\date{}

\begin{abstract}
We give an account of the calculations of the
$RO(Q)$-graded coefficient rings of some of the most basic
$Q$-equivariant cohomology theories, where  $Q$ is a group of order 2. One
 purpose is to advertise the effectiveness of the Tate square,
showing it has advantages over the slice spectral sequences in algebraically
simple cases. A second purpose is to give a single account showing
how to translate between the languages of different approaches. 
\end{abstract}

\thanks{I am grateful to R.R.Bruner, J.P.May
  and  L. Meier for collaborations which fed into this account, to the referee for objectivity
  about conventions, and to D.Dugger, D.C.Ravenel and Mingcong Zeng for pointing
  out an error in the  old version of the chart in Section 2.}
\maketitle

\tableofcontents
\newcommand{\Zu}{\underline{\Z}}
\newcommand{\rost}{\bigstar}
\newcommand{\Ftwo}{\mathbb{F}_2}
\newcommand{\kR}{k\R}
\newcommand{\vbh}{\hat{\vb}}

\section{Introduction}
Historically, complex orientable cohomology theories (complex
$K$-theory $KU$, complex cobordism $MU$, ...)  have been the
first to be exploited, partly because they tend to be easier to
calculate. Real analogues of these (real $K$-theory $KO$, real
cobordism $MO$, ... ) contain new information but can be hard to approach 
directly.  It is well known that representations  over the reals can be
viewed as complex representations with an additional conjugate-linear
involution: real representations are complex representations with an
action of the Galois group $Q=Gal(\C|\R)$ of order 2. 
 Atiyah used this technique \cite{Atiyah} to define  $K$-theory `with
 Reality' (or simply {\em Real} $K$-theory with the initial letter
 capitalized),  thereby obtaining a very practical approach to real
 $K$-theory by descent.  This was then extended to cobordism by Landweber
\cite{L}, Araki \cite{Araki} and others.  The resulting theories are
$Q$-equivariant cohomology theories represented by $Q$-spectra. 

The same discussion applies if we start with  $Q$-equivariant
theories:  the complex orientable theories (Atiyah-Segal
equivariant  $K$-theory, tom Dieck complex bordism) are much more
accessible. Many techniques were developed in the 1990s with
applications to equivariant complex orientable theories most clearly in mind. 

The point of this article is to describe some well known calculations for 
Real  equivariant theories in the simplest cases (ordinary cohomology
with constant integer coefficients and Real connective $K$-theory),
giving several different methods 
in each case. In particular, despite the fact that the $Q$-equivariant  Real theories are
not complex orientable (even though the underlying non-equivariant
theories are),  some of the complex orientable methods are still
extremely effective.

\subsection{Notation}
This subsection introduces our notation for some standard
machinery. I am grateful to the referee for patiently and objectively
pointing out ways in which  notational conventions vary.  Encouraged
by  the referee to clarify my conventions, the length of the subsection has 
increased about fivefold: I hope this will make the article more 
accessible, and ease communication between different notational
tribes. On the other hand, those familiar with the customs of my tribe 
can skip directly to Subsection \ref{subsec:Reality} now. 

\subsubsection{Grading} We write $\sigma$ for the sign representation
of $Q$ on $\R$, 1 for the trivial representation, and 
$\rho=1+\sigma$ for the real regular representation. We will be
grading our groups over the real
representation ring $RO(Q)=\{ x+y\sigma \st x,y\in \Z\}$.  We write
$M_*$ to indicate grading over $\Z$ and $M_{\rost}$ to denote grading
over  $RO(Q)$.  When we draw
pictures, they will be displayed in the plane with
the $\Z$ axis horizontal and the $\sigma$ axis vertical, so that
$x+y\sigma$ gets displayed at the point with cartesian coordinates
$(x,y)$. 

{\em Alternatives:}
We have followed \cite{HHR} in using $\sigma$ for the non-trivial real
representation of $Q$. In the past each author seems to have used a different
letter. For example, the present author previously used $\xi$  \cite{assiet, Tate, root, kobg}, 
the letter $\alpha$ is used by Hu and Kriz, and the letter $\tau$ is
used elsewhere. Perhaps we can unite around $\sigma$?

We have used  $RO(Q)$ grading both in notation and displays.
Atiyah \cite{Atiyah},  Landweber \cite{L}, Araki \cite{Araki} 
use $\R^{p,q}$ for $q\oplus p\sigma$ (sic). 
The Tate twist notation from algebraic geometry used by Dugger \cite{Dugger} is related to
$RO(Q)$-grading by  $\R^{p+q,q}=p+q\sigma. $

\subsubsection{Mackey functors} 
We will use Dress's formulation of Mackey functors for a finite group
$G$  \cite{Dress} as
given by a covariant and a contravariant functor on finite $G$-sets
subject  to the Mackey condition.  A Mackey functor $M$ for a group $G$ is 
therefore determined by its values on the transitive $G$-sets $G/H$
together with certain structure maps. 
If $K \subseteq H$ the projection $\pi_K^H: G/K\lra G/H$ induces a 
restriction map $(\pi_K^H)^*=\res^H_K: M(G/H)\lra M(G/K)$ and an
induction map $(\pi_K^H)_*=\ind_K^H: M(G/K)\lra M(G/H)$, and right
multiplication $R_g: G/H\lra G/H^g$ induces an action of
$W_G(H)=N_G(H)/H$ on $M(G/H)$. 

We write $\Zu$ for  the constant Mackey functor at $\Z$ (which is to
say that the values on all orbits are $\Z$,  restriction is the identity
and  induction is multiplication by the index of the smaller subgroup in
the larger one).

\subsubsection{Spectra} 
We will be working with  cohomology theories $E_G^*(\cdot)$, which are
represented by {\em genuine} $G$-spectra $E$ (i.e., $G$-spectra indexed on a
complete $G$-universe) in the sense that
$E_G^*(X)=[X,E]_G^*$. Accordingly the integer grading can be extended to an
$RO(G)$-grading. 

For an orthogonal  representation $V$ we write $S(V)$ for the unit
sphere and $S^V$ for the one point compactification of $V$ with $\infty$ as the basepoint.
We write $X^G_V:=\pi_V^G(X)=[S^V, X]^G$ for the $V$th $G$-equivariant homotopy 
group of $X$.
 As usual $X_V:=\pi_V(X)=[S^V,X]$ denotes the $V$th
non-equivariant homotopy group. The notation extends to virtual
representations.

The subcategory of $G$-spectra of the form $G/H_+$ is called the
{\em stable orbit category}, and calculation of the maps between them
shows that  they are generated by the maps $(\pi_K^H)_+$ and their duals,
and that a Mackey functor as described above is equivalent to 
a contravariant functor on the stable orbit category.
We write $\underline{\pi}^G_*(X)$ for the Mackey functor 
$[(\cdot)_+, X]^G_*$ whose value
on $G/H$ is $\pi_*^H(X)=[G/H_+,X]^G_*$. 

{\em Alternatives:}
The following is inserted in recognition of the fact that alternative
views are held very strongly by some (for example by the referee!). 

Some authors  (including the present author in the 1980s) 
prefer to omit the group when writing $RO(Q)$-grading and 
Mackey functors (writing $\underline{\pi}_*(X)$ where we write
$\underline{\pi}^Q_*(X)$  and $\pi_{\rost}(X)$ where we write
$\pi_{\rost}^Q(X)$). This is presumably on the grounds of brevity, but
it is at the cost of  forcing the adoption of a more complicated notation for non-equivariant
homotopy. The present choice  is based on  the ideas  (i) that working
equivariantly requires explicit acknowledgement whilst working
non-equivariantly does not and (ii) that one should be able
to substitute a particular grading for $\rost$ (so that considering
grading zero for instance, $\pi^Q_{\rost}(X)$ becomes $\pi^Q_0(X)$). 

\subsubsection{Ordinary cohomology} A cohomology theory $E_G^*(\cdot)$ satisfying the
dimension axiom (i.e., one whose values on cells $G/H_+$ is entirely in
degree 0) is called  {\em ordinary} (or {\em Bredon}). It is determined by the
values on these cells, which define a Mackey functor
$M=\underline{\pi}^G_0(E)$. We write $H^*_G(\cdot ; M)$ for the
cohomology theory and $HM$ for the representing $G$-spectrum. 

\subsubsection{Proto-Euler and Proto-Thom classes} 
For any ring spectrum  $E$,  the unit map provides, for any 
representation $V$ a class $\Sigma^V\iota \in E^{V}_G(S^V)$. A Thom class would be an
element $\tau_V \in E^{|V|}_G(S^V)$ which generates a rank one free
module. The Euler class (of degree $|V|-V$) would be the pullback of $\tau_V$
along the zero section $a_V: S^0\lra S^V$. 

With $G=Q$ some related classes  are important for
us. We write   $a=a_{\sigma}: S^0\lra S^\sigma$ for the inclusion, and think of it as  a homotopy
element of degree $-\sigma$. 

There are Euler-like classes $\lambda$ (of degree $1-\sigma$, which
does not survive to homotopy),  $u$ (of degree $2(1-\sigma)$, important for
ordinary cohomology), and  $U$ (of degree $4(1-\sigma)$, important for
$k\R$). 

{\em Alternatives:}
We follow \cite{HHR} in using $a$ for the degree $-\sigma$
Euler-like class. The letters $e$ and $i$ are used by some authors.

In \cite[Definition 3.4]{HHRCfour} the class $u$ is called
$u_{2\sigma}$, and the class $2u^{-1}$ is called $e_{2\sigma}$. 
In the more general study of bordism with reality, a class $u$ (of
which ours is the image) plays a significant role, and its $2^n$th
power is an almost-periodicity at the $n$th chromatic level. This
means that $U=u^2$, and that we should think of $\lambda$ as a 
square root of $u$.

\subsubsection{Basepoints} 
Our $Q$ spaces will be equipped with a $Q$-fixed basepoint and cohomology will be reduced.

\subsection{What's the point of this paper?}
The paper describes several different methods of calculating
the $RO(Q)$-graded homotopy of two well known Real spectra: the integral
Eilenberg-MacLane spectrum $H\Zu$ and connective $K$-theory with
reality $\kR$. The results and the methods are all known. The intention is
to describe several different methods in a way that lets one compare
them, to provide more details and more pictures than is usually
done, and to record the notations used by different authors. 
The author has found the need to make these comparisons quite
often, and that it takes time to pin details down precisely from the
existing literature. It is hoped the present summary may be useful to others. 

\subsection{Geography of the coefficients}
\label{subsec:Reality}

It is worth highlighting some features of the coefficient
rings $X^Q_{\rost}$. The background is that their non-equivariant
counterparts are in even degrees. Next the reader should look at
the pictures at the  the start of Sections \ref{sec:HZ}
and \ref{sec:kR}) (or at the picture of $tmf_1(3)^Q_{\rost}$ in
\cite[Subsection 13.C]{BPRn}). 

The crudest feature is that the non-zero entries are mostly above the
diagonal line of slope $-1$. Indeed the only exceptions to this are
the  $a$-multiple tails in the right half-plane. This is
characteristic of theories which are equivariantly connective. 

\begin{lemma}
If $X$ is a connective $Q$-spectrum then $X^Q_{x+y\sigma}$ is zero below
the antidiagonal $x+y=0$ in the half-plane $x<0$, and multiplication by $a$
is an isomorphism below the antidiagonal in the half-plane $x\geq 0$.
\end{lemma}

\begin{proof}
Non-equivariant connectivity and the cofibre sequence $Q_+\lra S^0\lra S^{\sigma}$ shows multiplication
by $a$ is surjective from the antidiagonal, and isomorphic below it. The
connectivity of $X^Q_*$ shows there are zeros along the negative $x$
axis. 
\end{proof}

However, the most striking feature is that the coefficients
of the theories described here have `The Gap': 
$$X^Q_{*\rho-i}=0 \mbox{  for } i =1, 2, 3 $$
so that  in the display of the $RO(Q)$-graded coefficients, the three
diagonals above the leading diagonal are all zero. It is worth
commenting on the implications of this. 

\begin{lemma}
(i) The vanishing of the diagonals $X^Q_{*\rho-i}$  with $i=1,2$ is equivalent to $X$  
being  {\em strongly even} in the sense of Meier
and Hill \cite{HM} (i.e.,  that  
$X^Q_{*\rho-1}=0$ and the restriction maps $X^Q_{k\rho}\lra X_{2k}$
are isomorphic for all $k$). 

(ii)  The vanishing of the three diagonals $X^Q_{*\rho-i}$ 
with $i=1, 2,3$ is equivalent to being strongly even and having
$X_*$ even. 
\end{lemma}

\begin{proof}
Consider the long exact sequence in $X$-cohomology induced by  the sequence  $Q_+\lra S^0\lra S^{\sigma}$. 
\end{proof}

\subsection{The Tate square}
We will make considerable use of the Tate square, so we briefly
describe the construction from \cite{GTate, Tate}. This is based on the 
contractible free $Q$-space $EQ$ and the join $\tilde{E}Q$ which are related
by the cofibre sequence
$$EQ_+\lra S^0\lra \tilde{E}Q$$
brought to prominence by Carlsson's use of it in his proof of the
Segal conjecture \cite{Carlsson}. 
Since $Q$ is a periodic group we use the model $EQ=S(\infty\sigma)$ for a contractible free $Q$-space and
$\tilde{E}Q=S^{\infty\sigma}$ for the join with $S^0$ to give periodic
resolutions as in \cite{GTate}.  

We will use Scandinavian notation: 
$$X^{hQ}=(X^h)^Q, \mbox{ where } X^h=F(EQ_+, X)$$
$$X_{hQ}\simeq  (X_h)^Q, \mbox{ where } X_h=EQ_+\sm X$$
$$X^{tQ}=(X^t)^Q, \mbox{ where } X^t=F(EQ_+, X)\sm \tilde{E}Q$$
$$X^{\Phi Q}=(X^\Phi)^Q, \mbox{ where } X^\Phi =X\sm\tilde{E}Q. $$
The notation $X^{\Phi Q}$ is justified by the fact that since $Q$ is
of prime order, it is the geometric $Q$-fixed point spectrum. 

It is easy to see that there are two linked cofibre sequences
$$\diagram
X_h\rto \dto_{\simeq} & X \rto \dto &X^{\Phi}\dto \\
X_h \rto &X^h\rto &X^t
\enddiagram$$
and the indicated equivalence arises since $X\lra X^h$ is a
non-equivariant equivalence. Accordingly the Tate square (the square
on the right) is a homotopy pullback, and it is a homotopy pullback of rings if $X$ is a
ring.

Typically we proceed as follows
\begin{itemize}
\item Calculate $X^{hQ}_{\rost}$ using the homotopy fixed point
  spectral sequence described in Subsection \ref{subsec:ROQhfpss}.
\item Infer $X^{tQ}_{\rost}$ by inverting $a$ (since
  $\tilde{E}Q=S^{\infty\sigma}$,  and evidently $\pi^Q_{\rost}(Y\sm S^{\infty
    \sigma})=\pi^Q_{\rost}(Y)[1/a]$).
\item Infer $X_{hQ}^{\rost}$ from the lower cofibre sequence (which
  amounts to a simple Local Cohomology Theorem).
\item Infer $X^{\Phi Q}_{\rost}$ from $X^{tQ}_{\rost}$ using Lemma
  \ref{lem:QcoverPhiQcover}; since both are $a$-periodic this follows
  from the integer graded homotopy groups. 
\item Infer $X^Q_{\rost}$ from the Tate square and the top sequence. 
\end{itemize}

The following 
 is immediate since the fibres of the two verticals in
the Tate square are equivalent. The hypothesis on the non-equivariant
spectrum shows $X_{hQ}$ is connective and hence it follows that $X^{\Phi Q}$ is connective.
\begin{lemma}
\label{lem:QcoverPhiQcover} (\cite[Lemma 11.2]{BPRn})
Suppose $X$ is a $Q$-spectrum which is non-equivariantly connective. 
If  $X^Q\lra X^{hQ}$ is the connective
cover then 
$X^{\Phi Q }\lra X^{tQ}$ is too. \qqed
\end{lemma}

\subsection{The $RO(Q)$-graded homotopy fixed point spectral sequence}
\label{subsec:ROQhfpss}

We note that we are trying to calculate the groups
$$X^{hQ}_{a+b\sigma}=[S^{a+b\sigma}, X^h]^Q=[S^{a+b\sigma}, F(EQ_+,
X)]^Q=[S^{a+b\sigma}\sm EQ_+, X]^Q=[S^a, F(S^{b\sigma}, X)^{hQ}]. $$

We could consider the separate homotopy fixed point spectral
sequences for the various spectra $F(S^{b\sigma}, X)$. To avoid
confusion later, we arrange this homologically from the start, so that it is in a
(vertically shifted) second quadrant:  
$$E_{s,t}^2(b)=H^{-s}(Q; X^{-t}(S^{b\sigma})))\Rightarrow
X^{-(s+t)}_Q(EQ_+\sm S^{b\sigma})=\pi_{s+t}(F(S^{b\sigma}, X)^{hQ})$$
with differentials
$$d^r_{s,t}: E^r_{s,t}(b) \lra E^r_{s-r, t+r-1}(b) . $$
Note that in the $E^2$-term $Q$ acts on $S^{b\sigma}$ by a degree
$(-1)^b$ map, and it is this that makes $X^*(S^{b\sigma})$ a
$Q$-module. 

However there is a great advantage to combining
these spectral sequences as $b$ varies, especially if $X$ is a ring spectrum, since the pairings 
$$S^{b_1\sigma}\sm S^{b_2\sigma}\lra S^{(b_1+b_2)\sigma}$$
give the whole construction a multiplicative structure. Although the
differentials all lie within the spectral sequence for a single value
of $b$, the spectral sequences all come from the same filtration, so
that the Leibniz rule applies across all the spectral sequences. 

\begin{defn}
The $RO(Q)$-graded homotopy fixed point spectral sequence  for $X$ is the trigraded
spectral sequence
$$E^2_{s,t,b}=H^{-s}(Q; X_{t+b\sigma})\Rightarrow [S^{s+t+b\sigma}\sm
EQ_+, X]^Q=:X^{hQ}_{s+t+b\sigma}$$
with differentials 
$$d^r_{s,t,b}: E^r_{s,t, b}\lra E^r_{s-r, t+r-1, b}. $$
\end{defn}

\section{Ordinary cohomology}
\label{sec:HZ}
We may depict the $RO(Q)$-graded coefficients of the Eilenberg-MacLane
spectrum, $H\Zu^Q_{\rost}$ as follows, where squares are copies of
$\Z$; circles are also copies of $\Z$, but the generator is twice what one might expect. 
Red dots are copies of $\Ftwo$. This was first calculated by
Stong \cite{Stong1, Stong2} and Waner \cite{Waner} and first published by Caruso \cite{Caruso}, but
it has appeared in several other places. The Mackey structure is made
explicit in \cite[Figure 2, Page 405]{HHRCfour}.

One notable feature is `The Gap', which constitutes the three lines of
slope 1 above the unit (i.e., the diagonals intersecting the $x$ axis
in $x=-3, -2, -1$). 

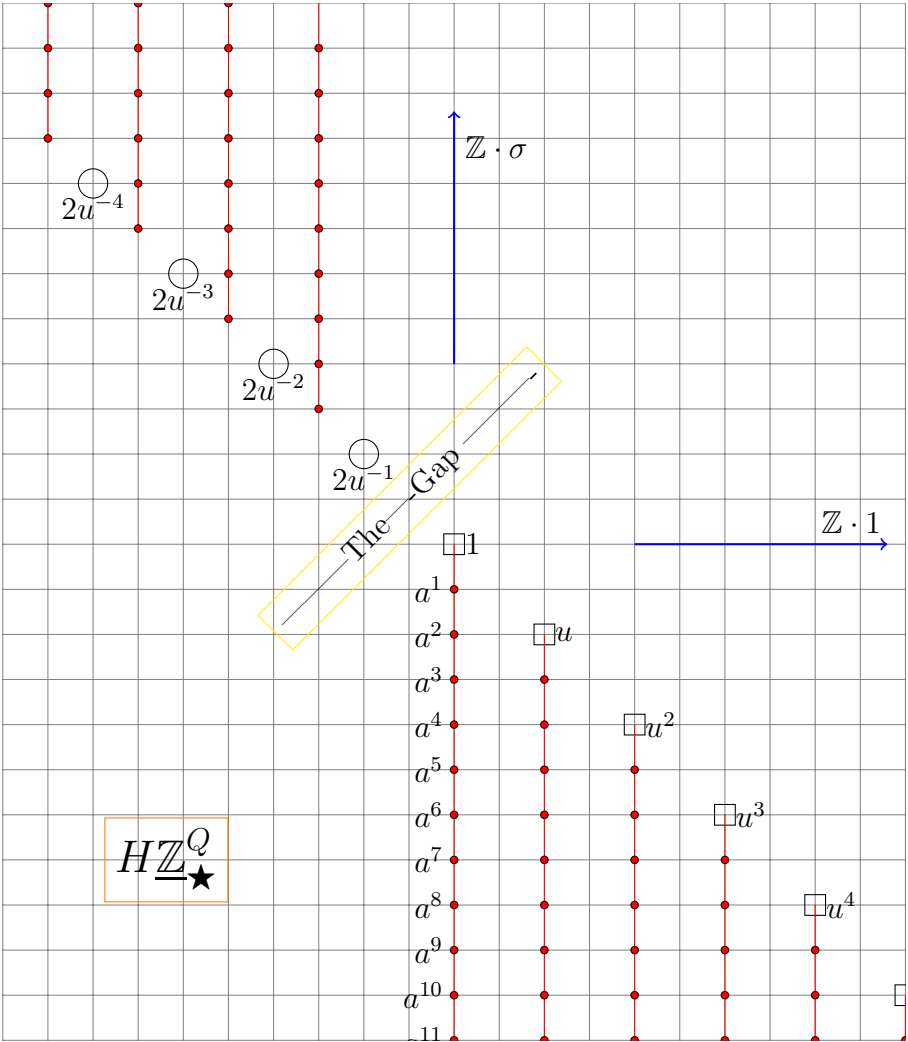
\begin{tikzpicture}
[scale =1.2]
\clip (-5, -5.5) rectangle (5, 6);
\draw[step=0.5, gray, very thin] (-9,-9) grid (9, 9);
\draw[blue, ->, thick](2,0)--(4.8, 0);
\draw (4.4, 0) node[anchor=south]{$\Z\cdot 1$};

\draw[blue, ->, thick](0,2)--(0, 4.8);
\draw (0,4.4) node[anchor=west]{$\Z\cdot \sigma$};

\draw (-2.5, -3.5) node[anchor=east,
draw=orange]{\Large{$H\Zu^Q_{\rost}$}};
\draw (1, 2) node[anchor=east, draw=yellow,
rotate=45]{{---------The----Gap ----------}};

\node at (0,0)  [shape = rectangle, draw]{};
\draw (0,0) node[anchor=west]{$1$};
\draw[red](0,0)--(0, -6);
\foreach \y in {1, 2, 3,4,5,6,7,8,9,10,11,12,13,14,15}
\draw (0,-\y/2) node[anchor=east] {$a^{\y}$};
\foreach \y in {1,2,3,4,5,6,7,8,9,10,11,12,13,14,15}
\node at (0,-\y/2) [fill=red, inner sep=1pt, shape=circle, draw] {};

\node at (1,-1)  [shape = rectangle, draw]{};
\draw (1, -1) node[anchor=west]{$u$};
\draw[red](1,-1)--(1, -6);
\foreach \y in {1,2,3,4,5,6,7,8,9,10,11,12,13,14,15}
\node at (1,-1-\y/2) [fill=red, inner sep=1pt, shape=circle, draw] {};

\node at (2,-2)  [shape = rectangle, draw]{};
\draw (2, -2) node[anchor=west]{$u^2$};
\draw[red](2,-2)--(2, -6);
\foreach \y in {1,2,3,4,5,6,7,8,9,10,11,12,13,14,15}
\node at (2,-2-\y/2) [fill=red, inner sep=1pt, shape=circle, draw] {};

\node at (3,-3)  [shape = rectangle, draw]{};
\draw (3, -3) node[anchor=west]{$u^3$};
\draw[red](3,-3)--(3, -6);
\foreach \y in {1,2,3,4,5,6,7,8,9,10,11,12,13,14,15}
\node at (3,-3-\y/2) [fill=red, inner sep=1pt, shape=circle, draw] {};

\node at (4,-4)  [shape = rectangle, draw]{};
\draw (4, -4) node[anchor=west]{$u^4$};
\draw[red](4,-4)--(4, -6);
\foreach \y in {1,2,3,4,5,6,7,8,9,10,11,12,13,14,15}
\node at (4,-4-\y/2) [fill=red, inner sep=1pt, shape=circle, draw] {};

\node at (5,-5)  [shape = rectangle, draw]{};
\draw (5, -5) node[anchor=west]{$u^5$};
\draw[red](5,-5)--(5, -6);
\foreach \y in {1,2,3,4,5,6,7,8,9,10,11,12,13,14,15}
\node at (5,-5-\y/2) [fill=red, inner sep=1pt, shape=circle, draw] {};

\node at (-1,1)  [shape = circle, draw]{};
\draw (-1,1) node[anchor=north]{$2u^{-1}$};
\draw[red](-1.5, 1.5) -- (-1.5, 6);
\foreach \y in {0, 1,2,3,4,5,6,7,8}
\node at (-1.5,1.5+\y/2) [fill=red, inner sep=1pt, shape=circle, draw] {};

\node at (-2,2)  [shape = circle, draw]{};
\draw (-2,2) node[anchor=north]{$2u^{-2}$};
\draw[red](-2.5, 2.5) -- (-2.5, 6);
\foreach \y in {0, 1,2,3,4,5,6,7,8}
\node at (-2.5,2.5+\y/2) [fill=red, inner sep=1pt, shape=circle, draw] {};

\node at (-3,3)  [shape = circle, draw]{};
\draw (-3,3) node[anchor=north]{$2u^{-3}$};
\draw[red](-3.5, 3.5) -- (-3.5, 6);
\foreach \y in {0, 1,2,3,4,5,6,7,8}
\node at (-3.5,3.5+\y/2) [fill=red, inner sep=1pt, shape=circle, draw]
{};

\node at (-4,4)  [shape = circle, draw]{};
\draw (-4,4) node[anchor=north]{$2u^{-4}$};
\draw[red](-4.5, 4.5) -- (-4.5, 6);
\foreach \y in {0, 1,2,3,4,5,6,7,8}
\node at (-4.5,4.5+\y/2) [fill=red, inner sep=1pt, shape=circle, draw] {};

\end{tikzpicture}

\subsection{Cell approach}
Note first that
$$H\Zu^Q_{a+b\sigma}=[S^{a+b\sigma},
H\Zu]=H^{-a}_Q(S^{b\sigma}; H\Zu). $$ 
Since $H\Zu$ satisfies the dimension axiom,  it is natural to calculate
this by finding a cell structure for $S^{b\sigma}$.

Indeed, we start from the cofibre sequence
$$Q_+\lra S^0 \lra S^{\sigma}. $$
Suspending by $n\sigma$ and using the untwisting isomorphism
$S^{n\sigma }\sm Q_+\simeq S^n \sm Q_+$,  this gives
$$S^n \sm Q_+ \lra S^{n\sigma} \lra S^{(n+1)\sigma}.$$

First  we consider a row in the lower half plane,
$H\Zu^Q_{*-n\sigma}$,  when $n\geq 0$. We argue as follows.
We begin with  $H^Q_*(S^0; \Zu)=\Z$ (by the dimension axiom). Adding an $(n+1)$-cell to deduce 
$H^Q_*(S^{(n+1)\sigma}; \Zu)$ from $H^Q_*(S^{n\sigma}; \Zu)$ we only
ever have to determine the map 
$$\Z =H^Q_n(S^n \sm Q_+; \Zu) \lra  H^Q_n(S^{n\sigma}; \Zu)=\Z$$
when $n$ is even. It is always multiplication by 2: 
for $n=0$ it is $\pi_*=\ind_1^Q: \Zu (Q/1)\lra \Zu (Q/Q)$. For larger even
$n$ we note inductively that 
$$H^Q_n(S^{n\sigma}; \Zu)\stackrel{\cong}\lra H^Q_n(S^{n\sigma}/S^{(n-1)\sigma};
\Zu)$$
is an isomorphism. The composite
$$S^n \sm Q_+\lra S^{n\sigma}\lra S^{n\sigma}/S^{(n-1)\sigma}\simeq
S^n \sm Q_+$$
is the $n$th suspension of the case $n=0$, where it is $\ind_1^Q$.

Next, consider a row in the upper half plane,  $H\Zu^Q_{*+n\sigma}$
when $n\geq 0$. We argue as follows. 
We begin with  $H_Q^*(S^0; \Zu)=\Z$ (by the dimension axiom). Adding an $(n+1)$-cell to deduce 
$H^*_Q(S^{(n+1)\sigma}; \Zu)$ from $H^*_Q(S^{n\sigma}; \Zu)$ we only
ever have to decide on the map 
$$\Z =H^n_Q(S^{n\sigma} ; \Zu) \lra  H^n_Q(S^{n}\sm Q_+; \Zu)=\Z$$
when $n$ is even. 
For $n=0$ it is $\pi^*=\res^Q_1: \Z (Q/Q)\stackrel{\cong}\lra \Z (Q/1)$. 
For larger even $n$ it is always multiplication by 2: 
 we note inductively that 
$$H^n_Q(S^{n\sigma}/S^{(n-1)\sigma})\lra H^n_Q(S^{n\sigma};\Zu)$$
is multiplication by $2$. The composite
$$S^n \sm Q_+\lra S^{n\sigma}\lra S^{n\sigma}/S^{(n-1)\sigma}\simeq
S^n \sm Q_+$$
is the $n$th suspension of the case $n=0$, where it is the sum $N=1+x$ 
over multiplication by group elements, where $Q=\langle x \rangle$.

\subsection{Quotient approach}
The cohomology of a {\em space} only depends on the restriction maps in the 
Mackey coefficients, and if these form a constant coefficient system (i.e.,
restriction maps are the identity) it follows that it is the
nonequivariant cohomology of the quotient: for any $Q$-space $X$
$$H^*_Q(X; \Zu)=H^*(X/Q; \Z). $$

For any $n\geq 0$ we have  $S^{(n+1)\sigma }=S^0* S(n\sigma)$, and hence
$$H^*_Q(S^{(n+1)\sigma}; \Zu)=H^*(S^0*\R P^n;\Z) .$$
This gives $H\Zu^Q_{*+n\sigma}$ for $n\geq 0$ from the well known
cohomology of projective spaces. This fills in the upper half plane. 

For the lower half plane we use the fact that for {\em free} $Q$-spectra 
$$H_*^Q(X; \Zu)=H_*(X/Q;\Z) .$$
This is not true  for arbitrary $Q$-spectra (or even for $Q$-spaces). 
We then use the cofibre sequence
$$S(n\sigma)_+ \lra S^0 \lra S^{n\sigma}$$
We have 
$$H_*^Q(S(n\sigma)_+; \Zu)=H_*(\R P^n_+;\Z) ;  $$
since $S^0$ only has homology in degree 0, we need only check that 
$$S(n\sigma)_+ \lra S^0$$
is multiplication by $2$ in degree 0. When $n=0$ this is part of the
definition of $\Zu$, and for higher dimensions, we use the fact that 
$S(\sigma)_+\lra S(n\sigma)_+$ is an isomorphism in degree $0$
since only higher cells have been attached. This fills in the lower
half plane. 

We note in passing that this also shows $H^Q_*(S^{n\sigma}; \Zu)\not
\cong H^Q_*(S^{n\sigma}/Q); \Z)$ for $n\geq 1$. 

\subsection{Tate approach}
The cleanest approach is to use the Tate diagram
$$\diagram
H\Zu_h\rto \dto_{\simeq} &H\Zu \rto \dto &H\Zu^{\Phi}\dto \\
H\Zu_h \rto &H\Zu^h\rto &H\Zu^t
\enddiagram$$

The procedure is to start with homotopy fixed points.  From this 
we can immediately deduce Tate cohomology by inverting $a$. This then
gives homotopy orbits by using the bottom row. These terms only depend
on the underlying spectrum (in this case $H\Zu$) made free. The geometric fixed points depends on the spectrum
itself, so needs further input. Then the answer follows from the Tate
square. 

\begin{lemma}
$$H\Zu^{hQ}_{\rost}= BB [u,u^{-1}], $$
where $BB=\Z [a]/(2a)$, $|a|=-\sigma, |u|=2-2\sigma$ ($a$ is the Euler
class of $\sigma$ and $u$ is a Thom class). 
\end{lemma}
\begin{remark}
\label{rem:groupcohom}
(i) Notice how this is related to ordinary group cohomology, which occurs
along the negative $x$-axis: 
$$H^*(BQ_+; \Z)=H\Zu^{hQ}_*=\Z [y]/(2y), \mbox{ where } y=a^2u^{-1}
\mbox{ is of codegree 2.} $$

(ii) The notation $BB$ stands for  Basic Block, and the usefulness of
the notation will emerge as we consider other examples.  
\end{remark}

\begin{proof}
We describe the $RO(Q)$-graded spectral sequence of Subsection
\ref{subsec:ROQhfpss} for $H\Zu$.
It is convenient to describe it from the $E_1$-term if we choose a good
cell structure  on $EQ_+$.  Indeed we use the filtration 
$$S(\sigma)_+\subset  S(2\sigma)_+ \subset  S(3\sigma)_+\subset  \cdots \subset
S(\infty \sigma)_+=EQ_+. $$
There is one free $Q$-cell in each degree and in cellular homology this gives the standard 2-periodic resolution 
$$0\lla \Z \stackrel{\eps}\lla \Z Q \stackrel{1-x}\lla \Z Q \stackrel{1+x}\lla \Z Q
\stackrel{1-x}\lla \cdots $$
of $\Z$ over the group ring $\Z Q$.

We may then describe the $E^1$ term conveniently as
$$E^1_{*,*,*}=\Z [a ] [\lambda, \lambda^{-1}]$$
where the $(s,t, b)$ gradings are 
$$|a|=(-1, 1, -1), |\lambda |=(0, 1, -1). $$
The differential is given by taking  $d_1a=0$ and 
$$d_1\lambda =2a.$$
This is immediate from the algebraic resolution since $Q$ acts
non-trivially on $H^b(S^{b\sigma})$ for
$b$ odd. 

It follows that $u=\lambda^2$ is a cycle and that there are no further
differentials. 
\end{proof}

We now immediately deduce the coefficients of the Tate theory by
inverting $a$, since $S^{\infty \sigma}\simeq S^0[1/a]$, and then the
homotopy orbits by taking local cohomology. 

\begin{cor}
The Tate cohomology is given by 
$$H\Zu^{tQ}_{\rost}=\Ftwo [a, a^{-1}] [u,u^{-1}], $$
and the coefficients of the homotopy orbits by
$$H\Zu_{hQ}^{\rost}=NB [u, u^{-1}]$$
where,   $NB=\Z \oplus \Sigma^{-1} \Ftwo [a]^{\vee}$. 
We note also that $NB\simeq R\Gamma_{(a)} (BB)$, i.e., it is the sum
of right derived functors of the $a$-power torsion functor
$\Gamma_{(a)}$. 
\end{cor}

\begin{remark}
The notation $NB$ stands for Negative Block, and as for $BB$ the
usefulness of the notation should emerge once we have other examples.  
\end{remark}

For geometric fixed points we apply Lemma \ref{lem:QcoverPhiQcover}.

\begin{lemma}
The geometric fixed points are given by 
$H\Zu^{\Phi Q}_{\rost}=\Ftwo [a,a^{-1}][u]. $
\end{lemma}

\begin{proof}
By definition of the Eilenberg-MacLane spectrum, $(H\Zu)^Q=H\Z$. It is
immediate that $(H\Zu)^{hQ}_*$ is zero in positive degrees, and the
algebraic resolution shows that $(H\Zu)^{Q}_0\stackrel{\cong}\lra (H\Zu)^{hQ}_0$ is an
isomorphism. The conclusion follows from Lemma \ref{lem:QcoverPhiQcover}.
\end{proof}

The calculation of $H\Zu^Q_{\rost}$ now concludes by using the Tate 
square. The only difference between $H\Zu^Q_{\rost}$ and the more
regular $H\Zu^{hQ}_{\rost}$ comes from the negative Tate $a$-columns.
The Tate square therefore  induces an actual  pullback square in gradings
with even integer-part. The $a$-divisible columns in odd integer
degrees $\leq -3$ come from the cokernel of $(H\Zu)^{hQ}_{\rost}\lra (H\Zu)^{tQ}_{\rost}$.

\begin{cor}
$$H\Zu^Q_{\rost}=BB  [u] \oplus u^{-1} \cdot NB [u^{-1}],  $$
where $BB=\Z [a]/(2a)$ as before and 
$$NB=\Z  \oplus \Sigma^{-\rho}\Ftwo [a]^{\vee}. $$
The $BB$-module structure  of $NB$ is  as suggested by the notation,
and multiplication by $u$ relates factors of $NB$ and $BB$ according
to the notation,  with  the map $u^{-1} \cdot NB \lra BB$
being the composite $NB\lra \Z \lra BB$.
\end{cor}

\section{$K$-theory with reality}
\label{sec:kR}
We now repeat the exercise with the connective version of Atiyah's
$K$-theory with Reality: $k\R$. Again squares are copies of
$\Z$ and dots are copies of $\Ftwo$; circles are also copies of $\Z$, 
but the generator is twice what one might expect. 
This calculation
first appears in \cite{Dugger}, and there are  alternative approaches
in \cite{kobg, BPRn}. 

One notable feature is `The Gap', which constitutes the three lines of
slope 1 above the unit (i.e., the diagonals intersecting the $x$ axis
in $x=-3, -2, -1$.). In the periodic form of the theory (i.e., with
the degree $\rho$ class $\vb$ inverted) the
$a$-columns vanish and The Gap is copied across the page with the $U$-periodicity.

\begin{tikzpicture}
[scale =1.2]
\clip (-5, -5.5) rectangle (5, 6);
\draw[step=0.5, gray, very thin] (-9,-9) grid (9, 9);

\draw[blue, ->, thick](1.6,0)--(4.8, 0);
\draw (2.9, 0) node[anchor=north]{$\Z\cdot 1$};

\draw[blue, ->, thick](0,1.65)--(0, 5.4);
\draw (0,4.5) node[anchor=east]{$\Z\cdot \sigma$};

\draw (-2.5, -3.5) node[anchor=east,
draw=orange]{\Large{$k\R^Q_{\rost}$}};
\draw (1, 2) node[anchor=east, draw=yellow,
rotate=45]{{---------The----Gap ----------}};

\draw [->] (0,0)-- (8, 8);
\node at (0,0)  [shape = rectangle, draw]{};
\draw (0,0) node[anchor=east]{1};

\draw [->](1,-1)-- (9, 7);
\node at (1,-1) [shape=circle, draw] {};
\draw (1,-1) node[anchor=east]{$2u$};
\draw (-1,1) node[anchor=east]{$2u^{-1}$};

\draw[purple] (0,0) -- (0,-8);

\draw[red] (0,-0.5)-- (8, 7.5);
\draw[red] (0,-1)-- (8, 7);
\foreach \y in {1, 2, 3,4,5,6,7,8,9,10,11,12,13,14,15}
\draw (0,-\y/2) node[anchor=east] {$a^{\y}$};
\foreach \y in {1,2,3,4,5,6,7,8,9,10,11,12,13,14,15}
\node at (0,-\y/2) [fill=red, inner sep=1pt, shape=circle, draw] {};

\draw[pink] (0,0)-- (1,-1);

\draw [->] (2,-2)-- (8, 4);
\node at (2,-2)  [shape = rectangle, draw]{};
\draw (2,-2) node[anchor=east]{$U$};
\draw (-2,2) node[anchor=east]{$2U^{-1}$};

\draw [->](3,-3)-- (9, 3);
\node at (3,-3) [shape=circle, draw] {};
\draw (3,-3) node[anchor=east]{$2u^3$};
\draw (-3,3) node[anchor=east]{$2u^{-3}$};

\draw[purple] (2,-2) -- (2,-8);

\draw[red] (2,-2.5)-- (7, 2.5);
\draw[red] (2,-3)-- (7, 2);
\foreach \y in {1,2,3,4,5,6,7,8,9,10,11,12,13,14,15}
\node at (2,-2-\y/2) [fill=red, inner sep=1pt, shape=circle, draw] {};

\draw[pink] (2,-2)-- (3,-3);

\draw [->] (4,-4)-- (8, 0);
\node at (4,-4)  [shape = rectangle, draw]{};
\draw (4,-4) node[anchor=east]{$U^2$};
\draw (-4,4) node[anchor=east]{$2U^{-2}$};

\node at (5,-5) [shape=circle, draw] {};

\draw[purple] (4,-4) -- (4,-8);

\draw[red] (4,-4.5)-- (7, -1.5);
\draw[red] (4,-5)-- (7, -2);
\foreach \y in {1,2,3,4,5,6,7,8,9,10,11,12,13,14,15}
\node at (4,-4-\y/2) [fill=red, inner sep=1pt, shape=circle, draw] {};

\draw[pink] (4,-4)-- (5,-5);


\draw (-2,2)-- (5, 9);
\draw (4,8) node[anchor=east]{$(2, \overline{v}_1)P$};
\node at (-2,2) [shape=circle, draw] {};
\node at (-1.5,2.5) [shape=rectangle, draw] {};

\draw (-1,1)-- (7, 9);
\node at (-1,1) [shape=circle, draw] {};

\foreach \y in {0, 1,2,3,4,5,6,7,8}
\node at (-2.5,2.5+\y/2) [fill=red, inner sep=1pt, shape=circle, draw]
{};
\draw[red] (-2.5, 2.5) --(-2.5, 8);

\node at (-1.5,1.5) [fill=red, inner sep=1pt, shape=circle, draw] {};
\draw[red] (-1.5,1.5)-- (5, 8);
\node at (-1.5,2.0) [fill=red, inner sep=1pt, shape=circle, draw] {};
\draw[red] (-1.5,2.0)-- (5, 8.5);

\draw (-4,4)-- (5, 13);
\node at (-4,4) [shape=circle, draw] {};
\node at (-3.5,4.5) [shape=rectangle, draw] {};

\draw (-3,3)-- (7, 13);
\node at (-3,3) [shape=circle, draw] {};

\foreach \y in {0, 1,2,3,4,5,6,7,8}
\node at (-4.5,4.5+\y/2) [fill=red, inner sep=1pt, shape=circle, draw]
{};
\draw[red] (-4.5, 4.5) --(-4.5, 8);

\node at (-3.5,4) [fill=red, inner sep=1pt, shape=circle, draw] {};
\draw[red] (-3.5,4)-- (5, 12.5);
\node at (-3.5,3.5) [fill=red, inner sep=1pt, shape=circle, draw] {};
\draw[red] (-3.5,3.5)-- (5, 12);

\end{tikzpicture}

\subsection{Adding cells}
The approach through a cell structure of $S^{b\sigma}$  is a fairly
unattractive since the value of the theory on cells is complicated: 
$k\R^Q_*=ko_*$.   On the other hand, it is effectively what was done  in
\cite{kobg},  by working up the Tate filtration, so we will not repeat it here.

\subsection{Quotient approach}
The good behaviour on quotients is only easy to see on free spectra,
so the content can be extracted from the Tate approach by truncating appropriately. 
Again, it does not seem worth repeating here. 

\subsection{Tate approach}
The cleanest approach is to use the Tate diagram
$$\diagram
k\R_h\rto \dto_{\simeq} &k\R \rto \dto &k\R^{\Phi}\dto \\
k\R_h \rto &k\R^h\rto &k\R^t
\enddiagram$$

The only real calculational input is the homotopy fixed points.
\begin{lemma}
The homotopy fixed point coefficients are 
$$k\R^{hQ}_{\rost}=BB [U,U^{-1}], $$
where $BB=\Z [a, \vb]/(2a, \vb a^3)\oplus 2U\cdot \Z[\vb]$ is the
Basic Block for $k\R$. Here $|a|=-\sigma$ as before and
$|\vb|=1+\sigma$; the basic block $BB$ is copied across the page by the periodicity operator
$U$ of degree  $ |U|=4-4\sigma$. 
\end{lemma}

\begin{remark}
Notice how this is related to ordinary group cohomology, which occurs
along the negative $x$-axis: 
$$\kR^{hQ}_{\leq 0}=\Z [Y]/(2Y), \mbox{ where } Y=a^4U^{-1}
\mbox{ is of codegree 4.} $$
The element $Y$ corresponds to $y^2$ from $H\Zu^{hQ}_*$ in Remark
\ref{rem:groupcohom}. 

Similarly $ko_*=(k\R)^{hQ}_{\geq 0}$ occurs along the positive $x$-axis,
with $\eta=a\vb$ generating $ko_1$, $2u\vb^2$ generating $ko_4$ and the
Bott element  $U \vb^4$ generating $ko_8$.
\end{remark}

\begin{proof}
We describe the $RO(Q)$-graded spectral sequence of Subsection
\ref{subsec:ROQhfpss} for $k\R$. In fact it
is convenient to describe it from the $E_1$-term if we choose a good
cell structure  on $EQ_+$.  Indeed we use the filtration 
$$S(\sigma)_+\subset  S(2\sigma)_+ \subset  S(3\sigma)_+\subset  \cdots \subset
S(\infty \sigma)_+=EQ_+. $$
There is one free $Q$-cell in each degree and in cellular homology this gives the standard 2-periodic resolution 
$$0\lla \Z \stackrel{\eps}\lla \Z Q \stackrel{1-x}\lla \Z Q \stackrel{1+x}\lla \Z Q
\stackrel{1-x}\lla \cdots $$
of $\Z$ over the group ring $\Z Q$.

We may then describe the $E^1$ term conveniently as
$$E^1_{*,*,*}=\Z [a, \vb] [\lambda , \lambda^{-1}]$$
where the $(s,t, b)$ gradings are 
$$|a|=(-1, 1, -1), |\lambda|=(0, 1, -1), |\vb|=(0, 1, 1). $$
The differential is given by making $a,\vb$ into $d_1$ cycles and 
$$d_1\lambda=2a.$$
This is clear from the algebraic resolution since we know $Q$ acts to
negate odd powers of $v$ and is of degree $-1$ on $S^{b\sigma}$ for
$b$ odd. 

This means that $u=\lambda^2$ is a $d_1$-cycle and we may calculate
$$E^2_{*,*,*}=\left\{ \Z [a]/(2a)\right\} [\vb] [u, u^{-1}]. $$
The next differential is the one piece of topological input. It is
forced by the fact that $\eta^4=0$ (from the stable homotopy of
$S^0$). 
The differential is given by making $a, \vb$ into cycles again and  
$$d_2u=\vb a^3. $$
This makes $U=u^2$ into a $d_2$-cycle and 
$$E^3=\left\{  \Z [a, \vb]/((2a, \vb a^3) \oplus (2u) \cdot \Z [\vb]\right\}
[U, U^{-1}],  $$
where the relations  $(2u)\cdot a=0$ and $(2u)^2=4u^2=4U$ are left implicit. 
Now we find that all subsequent differentials on $a, \vb $ or $U$ have
target in a zero group, so the spectral sequence has collapsed. 
\end{proof}

Because $S^{\infty \sigma}\simeq S^0[1/a]$ it is then easy to calculate Tate
cohomology by inverting $a$ and homotopy orbits as local cohomology. 

\begin{cor}
$$k\R^{tQ}_{\rost}=\Ftwo [a, a^{-1}] [U,U^{-1}], $$
$$k\R_{hQ}^{\rost}=NB [U,U^{-1}],   $$
where $NB=(2, \vb  ) \oplus \Sigma^{-1}\Ftwo [a]^{\vee}$ is the
Negative Block for $k\R$, more
conceptually described as the right derived functors of the $a$-power
torsion functor $\Gamma_{(a)}$: 
$$NB=R\Gamma_{(a)}(BB). $$
\end{cor}

For geometric fixed points we apply  Lemma \ref{lem:QcoverPhiQcover}.

\begin{lemma}
The geometric fixed points are given by $k\R^{\Phi Q}_{\rost}=\Ftwo [a,a^{-1}][U]$.
\end{lemma}

\begin{proof}
By construction $KO=(K\R)^Q$. Since $(K\R)^{tQ}\simeq *$ (by
nilpotence of $a$ in $(K\R)^{hQ}_{\rost}$) we see
$(K\R)^Q\simeq (K\R)^{hQ}$. Now the behaviour of the homotopy fixed
point spectral sequence for $K\R$ shows that of $k\R$ and hence that 
$(k\R)^Q\lra (k\R)^{hQ}$ is the connective cover. We may therefore
apply Lemma \ref{lem:QcoverPhiQcover} to obtain the conclusion. 
\end{proof}

Finally, we return to the Tate square, and read off the conclusion:
the only difference between $k\R^Q_{\rost}$ and the more regular 
$k\R^{hQ}_{\rost}$ comes from the negative Tate $a$-columns.

\begin{cor}
The $RO(Q)$-graded homotopy of $k\R$ is as follows
$$k\R^Q_{\rost}=BB[U]\oplus U^{-1}\cdot 
NB [U^{-1}],  $$
where the action of $U$ across the boundary 
$$U^{-1} \cdot NB \lra BB$$
factors out the $H^1_{(a)}$ towers and includes the $a$-power
torsion in $BB$.
\end{cor}

\subsection{The Bockstein spectral sequence}
Dugger \cite{Dugger} has shown that there is a cofibre sequence
$$\Sigma^\rho k\R \stackrel{\vb}\lra \kR \lra H\Zu  $$
where $\rho=1+\sigma$. (This can be proved by noting the homotopy
fixed point spectral sequence gives an element $\vb$, and slice
filtration methods easily show the mapping cone is $H\Zu$).
We  use the Dugger sequence  to calculate $\kR^Q_{\rost}$ from $H\Zu^Q_{\rost}$ by the Bockstein
spectral sequence.  This spectral sequence happens to coincide with
the slice spectral sequence for $k\R$, as we make explicit in
Subsection \ref{subsec:slice},  but we will not
make use of this. 

Accordingly (using the conventions of \cite[Section 4.1]{kobg}) we consider the spectral sequence obtained by taking
$RO(Q)$-graded homotopy of the diagram 
$$\diagram 
\kR \dto &\Sigma^{\rho} \kR \lto \dto &\Sigma^{2\rho} \kR \lto \dto
&\Sigma^{3\rho} \kR \lto \dto &\lto \cdots\\
\Z & \Sigma^{\rho}H\Zu & \Sigma^{2\rho}H\Zu & \Sigma^{3\rho}H\Zu&
\enddiagram$$
to obtain a spectral sequence
$$E^1=H\Zu^Q_{\rost}[\vbh] \Rightarrow \kR^Q_{\rost}$$
where $\vbh$ is a formal variable of bidegree $(1, \rho)$. Since we
have already calculated $H\Zu^Q_{\rost}$ this is straightforward once
we have found a reasonable way to display the information. The
standard thing to do in the $\Z$-graded case (for $\Sigma ko \lra ko
\lra ku $ for example) is that for each row of
constant $s$ we write $ku_*$ horizontally, shifted to the right by $s$,  so that
the differential $d_r$ takes on the Adams form of going up $r$ rows
and one column to the left (i.e., subtracting 1).  

In our present context, for each fixed $s$ we have a whole
$RO(Q)$-plane. Accordingly we will refer to this counterpart of the 
$s$th row as the $s$th {\em floor}.  Once again we
arrange that $d^r$ goes up from the $s$th floor to the $(s+r)$th
floor. If we place $\pi^Q_{\rost}(\Sigma^{s\rho}H\Zu)$ on the $s$th floor, $d^r$ again
subtracts 1: symbolically, for any $\alpha \in
RO(Q)$, 
$$d^r: H\Zu^Q_{\alpha }\lra  (\Sigma^{r\rho}H\Zu)^Q_{\alpha
  -1}=H\Zu^Q_{\alpha -r\rho -1}=H\Zu^Q_{\alpha -r-1-r\sigma}
$$
The value on $x$ is defined by finding the  image of $x$  in $\kR^Q_{\alpha-1}$,
dividing by $\vb^{r-1}$ and looking at the image in $H\Zu^Q_{\rost}$
again.  We need only determine it in  a few cases and then propogate
this using the multiplicative structure.

For those with old-fashioned habits, it is natural to arrange the
calculation conventionally in the upper half-plane with the box at
$(n,s)$ containing $H\Zu^Q_{n-s+*\sigma}$, then $d^r$ relates $(n,s)$
to  $(n-1, s+r)$ and reduces the $\sigma$-grading by $r$. We then 
 find the value of $ku^Q_{n+*\sigma}$ by adding up the terms for fixed  $n$

On the other hand this places a strain on visualization, and an
alternative is to pile copies of $H\Zu^Q_{\rost}$ vertically above
each other and look from above. Each $H\Zu^Q_{a+b\sigma}$ that you see
then represents a whole pile,   $H\Zu^Q_{a+b\sigma}[\vbh]$, of copies.
 Then the differential $d^r$ raises the floor in the pile by $r$ and 
as indicated above it reduces $a+b\sigma$ by $(r+1)+r\sigma$. Finding the
answer then involves adding up entries from different stacks with
total degree constant. 

We illustrate this for $d^1$, which has degree $-2-\sigma$: 
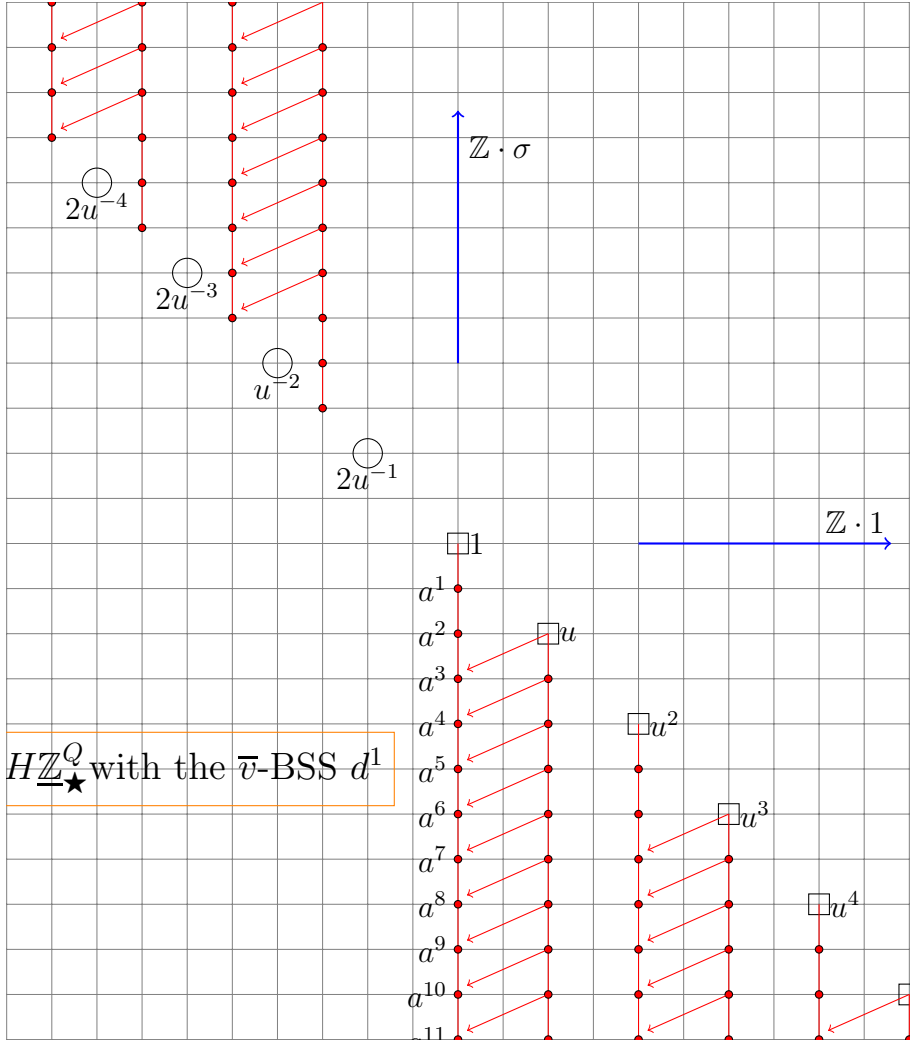
$$\begin{tikzpicture}
[scale =1.2]
\clip (-5, -5.5) rectangle (5, 6);
\draw[step=0.5, gray, very thin] (-9,-9) grid (9, 9);
\draw[blue, ->, thick](2,0)--(4.8, 0);
\draw (4.4, 0) node[anchor=south]{$\Z\cdot 1$};

\draw[blue, ->, thick](0,2)--(0, 4.8);
\draw (0,4.4) node[anchor=west]{$\Z\cdot \sigma$};

\draw (-0.7, -2.5) node[anchor=east,
draw=orange]{\large{$H\Zu^Q_{\rost}$\\ with the $\vb$-BSS $d^1$}};

\node at (0,0)  [shape = rectangle, draw]{};
\draw (0,0) node[anchor=west]{$1$};
\draw[red](0,0)--(0, -6);
\foreach \y in {1, 2, 3,4,5,6,7,8,9,10,11,12,13,14,15}
\draw (0,-\y/2) node[anchor=east] {$a^{\y}$};
\foreach \y in {1,2,3,4,5,6,7,8,9,10,11,12,13,14,15}
\node at (0,-\y/2) [fill=red, inner sep=1pt, shape=circle, draw] {};

\node at (1,-1)  [shape = rectangle, draw]{};
\draw (1, -1) node[anchor=west]{$u$};
\draw[red](1,-1)--(1, -6);
\foreach \y in {1,2,3,4,5,6,7,8,9,10,11,12,13,14,15}
\node at (1,-1-\y/2) [fill=red, inner sep=1pt, shape=circle, draw] {};
\foreach \y in {0, 1,2,3,4,5,6,7,8}
\draw[->, red](1,-5.0+\y/2)--(0.1, -5.4+\y/2);

\node at (2,-2)  [shape = rectangle, draw]{};
\draw (2, -2) node[anchor=west]{$u^2$};
\draw[red](2,-2)--(2, -6);
\foreach \y in {1,2,3,4,5,6,7,8,9,10,11,12,13,14,15}
\node at (2,-2-\y/2) [fill=red, inner sep=1pt, shape=circle, draw] {};

\node at (3,-3)  [shape = rectangle, draw]{};
\draw (3, -3) node[anchor=west]{$u^3$};
\draw[red](3,-3)--(3, -6);
\foreach \y in {1,2,3,4,5,6,7,8,9,10,11,12,13,14,15}
\node at (3,-3-\y/2) [fill=red, inner sep=1pt, shape=circle, draw] {};
\foreach \y in {0, 1,2,3,4,5,6,7,8}
\draw[->, red](3,-7+\y/2)--(2.1, -7.4+\y/2);

\node at (4,-4)  [shape = rectangle, draw]{};
\draw (4, -4) node[anchor=west]{$u^4$};
\draw[red](4,-4)--(4, -6);
\foreach \y in {1,2,3,4,5,6,7,8,9,10,11,12,13,14,15}
\node at (4,-4-\y/2) [fill=red, inner sep=1pt, shape=circle, draw] {};

\node at (5,-5)  [shape = rectangle, draw]{};
\draw (5, -5) node[anchor=west]{$u^5$};
\draw[red](5,-5)--(5, -6);
\foreach \y in {1,2,3,4,5,6,7,8,9,10,11,12,13,14,15}
\node at (5,-5-\y/2) [fill=red, inner sep=1pt, shape=circle, draw] {};
\foreach \y in {0, 1,2,3,4,5,6,7,8}
\draw[->, red](5,-9+\y/2)--(4.1, -9.4+\y/2);

\node at (-1,1)  [shape = circle, draw]{};
\draw (-1,1) node[anchor=north]{$2u^{-1}$};
\draw[red](-1.5, 1.5) -- (-1.5, 6);
\foreach \y in {0, 1,2,3,4,5,6,7,8}
\node at (-1.5,1.5+\y/2) [fill=red, inner sep=1pt, shape=circle, draw] {};
\foreach \y in {0, 1,2,3,4,5,6,7,8}
\draw[->, red](-1.5,3+\y/2)--(-2.4, 2.6+\y/2);

\node at (-2,2)  [shape = circle, draw]{};
\draw (-2,2) node[anchor=north]{$u^{-2}$};
\draw[red](-2.5, 2.5) -- (-2.5, 6);
\foreach \y in {0, 1,2,3,4,5,6,7,8}
\node at (-2.5,2.5+\y/2) [fill=red, inner sep=1pt, shape=circle, draw] {};

\node at (-3,3)  [shape = circle, draw]{};
\draw (-3,3) node[anchor=north]{$2u^{-3}$};
\draw[red](-3.5, 3.5) -- (-3.5, 6);
\foreach \y in {0, 1,2,3,4,5,6,7,8}
\node at (-3.5,3.5+\y/2) [fill=red, inner sep=1pt, shape=circle, draw]
{};
\foreach \y in {0, 1,2,3,4,5,6,7,8}
\draw[->, red](-3.5,5+\y/2)--(-4.4, 4.6+\y/2);

\node at (-4,4)  [shape = circle, draw]{};
\draw (-4,4) node[anchor=north]{$2u^{-4}$};
\draw[red](-4.5, 4.5) -- (-4.5, 6);
\foreach \y in {0, 1,2,3,4,5,6,7,8}
\node at (-4.5,4.5+\y/2) [fill=red, inner sep=1pt, shape=circle, draw] {};
\end{tikzpicture}$$
\newcommand{\thb}{\overline{\theta}}
To justify the displayed $d^1$, note that it is induced by a map $\theta :
H\Zu\lra \Sigma^{2+\sigma}H\Zu$, which is the analogue of the integral
lift of $Sq^3$ in the non-equivariant setting. The stable operations
of $H\Zu$ are relatively complicated, but in our case it suffices to
observe that $\theta $ induces a map $\thb: b\lra \Sigma^3 b$ where $b=F(EQ_+,
H\Z/2)$ represents mod 2 Borel cohomology. The stable operations in
this case are easy to determine: we immediately see $b^*_Qb=H^*(BQ_+; \Z/2)\tensor
\mathcal{A}^*$ (the operations are studied further in \cite{cc}). In
any case,  we see that for representations $V$
$$\theta: H^n_Q(S^V; \Zu)= H\Zu^Q_{V-n}\lra
H\Zu^Q_{V-n-2-\sigma}=H^{n+2}_Q(S^{V-\sigma})$$
is 
$$\theta: H^{n-1}_Q(S(V); \Zu)\lra H^{n+1}_Q(S(V-\sigma);\Zu)$$
for $n\geq 2$ and provided $V$ contains $\sigma$. When $n-1$ is even,
reduction mod $2$ is an isomorphism and this agrees with 
$$\thb: b^{n-1}_Q(S(V))\lra b^{n+1}_Q(S(V-\sigma))=b^{n+2}_Q(S(V)). $$
Such an operation is a linear combination  $\lambda Sq^3+\mu Sq^2Sq^1$.
We need only argue that $\lambda \neq 0$, since this forces $\mu=0$
(so as to get $(d^1)^2=0$).  There are several approaches to understanding
this operation. The first is to argue from the fact that  $\eta^4=0$
(this involves  looking at the whole BSS as displayed in Figure 1 below).
The second is to study operations sufficiently to see
the behaviour of $d^1$ is forced by the non-equivariant
situation. Perhaps the most elementary  is
to begin by understanding the situation in geometric fixed points
(i.e., after inverting $a$).  After the fact, we know that the cofibre sequence 
$$\Sigma \kR^{\Phi Q}=(\Sigma^{\rho}\kR)^{\Phi Q}\lra \kR^{\Phi Q}\lra H\Zu^{\Phi Q}$$
has $H\Zu^{\Phi Q}_*=\Ftwo [x]$ where $x$ is of degree 2
and  $\kR^{\Phi Q}_*=\Ftwo [x^2]$. However, we do not need to use this full analysis of
the cofibre sequence to  see that the operation is
non-trivial from the 2 column to the 0 column:  we only need to know 
that $\kR^{\Phi Q}_1=0$, which is easily checked from the Dugger
sequence  and known values of $H\Zu^Q_{\rost}$.  This then determines
the operation, which gives all the values. Using any one of these
methods  shows that  $d^1$ is as displayed.


We could now jump straight to  displaying $E^2=E^{\infty}$. 
One simply notes that the $s=0$ line consists of kernels of $d^1$ and
each of the higher rows consists of the homology (with the $s$th row
shifted $s$ steps to the right).  

Alternatively, instead of dealing with $d^1$ in piles, we could have used 
the traditional display, and we will describe this next. 
 Each box for constant $(n,s)$ contains a group graded on multiples of
$\sigma$. 
Because of the almost-periodicity of $H\Zu^Q_{\rost}$ of $2(1-\sigma)$
and for typographical reasons we have arranged that $n(1-\sigma)$ is 
the midpoint of the $(n,s)$-box. In other words, the element $x$ of
$(n,s)$-degree $(2,0)$ is represented by horizontal translation and
corresponds  to $u$ of degree $2(1-\sigma)$. 

We will now display the $E^1$-term in Figure 1 and then the $E^2=E^{\infty}$
term in Figure 2. The squares and circles denote copies of $\Z$ and the dots
denote copies of $\Ftwo$. The larger (pale blue) blobs pick out the
slice spectral sequence as detailed in the next section. 

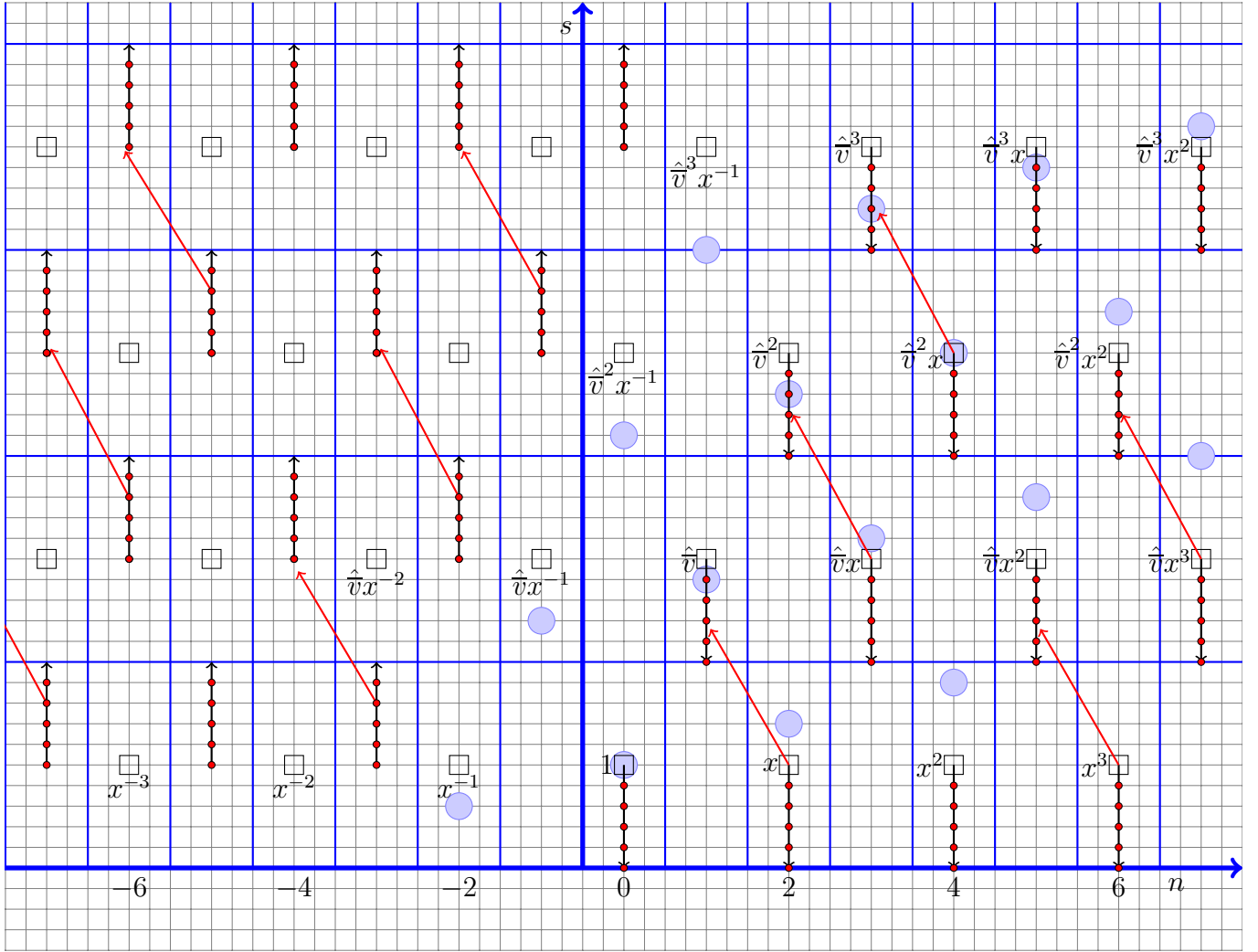
\begin{figure}
\begin{tikzpicture}
[scale =1.2]
\clip (-5, -5.5) rectangle (10, 6);
\draw[step=0.25, gray, very thin] (-9,-9) grid (15, 9);

\draw[blue, ->, line width=0.7mm](-10,-4.5)--(10, -4.5);
\draw (9.2, -4.5) node[anchor=north]{$n$};
\draw[blue, thick](-10,-2)--(10, -2);
\draw[blue, thick](-10,0.5)--(10, 0.5);
\draw[blue, thick](-10,3.0)--(10, 3.0);
\draw[blue, thick](-10,5.5)--(10, 5.5);
\draw (-3.5,-4.5) node[anchor=north]{$-6$};
\draw (-1.5,-4.5) node[anchor=north]{$-4$};
\draw (0.5,-4.5) node[anchor=north]{$-2$};
\draw (2.5,-4.5) node[anchor=north]{$0$};
\draw (4.5,-4.5) node[anchor=north]{$2$};
\draw (6.5,-4.5) node[anchor=north]{$4$};
\draw (8.5,-4.5) node[anchor=north]{$6$};

\draw[blue, ->, line  width=0.7mm](2,-4.5)--(2, 6);
\draw (2,5.5) node[anchor=south east]{$s$};
\draw[blue,  thick](-5,-4.5)--(-5, 6);
\draw[blue,  thick](-4,-4.5)--(-4, 6);
\draw[blue,  thick](-3,-4.5)--(-3, 6);
\draw[blue,  thick](-2,-4.5)--(-2, 6);
\draw[blue,  thick](-1,-4.5)--(-1, 6);
\draw[blue,  thick](0,-4.5)--(0, 6);
\draw[blue,  thick](1,-4.5)--(1, 6);
\draw[blue,  thick](3,-4.5)--(3, 6);
\draw[blue,  thick](4,-4.5)--(4, 6);
\draw[blue,  thick](5,-4.5)--(5, 6);
\draw[blue,  thick](6,-4.5)--(6, 6);
\draw[blue,  thick](7,-4.5)--(7, 6);
\draw[blue,  thick](8,-4.5)--(8, 6);
\draw[blue,  thick](9,-4.5)--(9, 6);

\foreach \y in {-1}
\node at (10.5 +\y, 2.75 + 2.25 * \y ) [circle, draw=blue!50,
fill=blue!20] {}; 

\foreach \y in {-2, -1, 0, 1}
\node at (8.5 +\y, 2.25 + 2.25 * \y ) [circle, draw=blue!50,
fill=blue!20] {}; 

\foreach \y in {-2, -1, 0, 1}
\node at (6.5 +\y, 1.75 + 2.25 * \y ) [circle, draw=blue!50,
fill=blue!20] {}; 

\foreach \y in {0, 1, 2 ,3}
\node at (2.5 +\y, -3.25 + 2.25* \y ) [circle, draw=blue!50,fill=blue!20] {}; 

\foreach \y in {0, 1, 2,3}
\node at (0.5 +\y, -3.75 + 2.25 * \y ) [circle, draw=blue!50, fill=blue!20] {}; 




\draw [->, thick] (2.5,-3.25)-- (2.5, -4.5);
\node at (2.5,-3.25)  [shape = rectangle, draw]{};
\draw (2.5,-3.25) node[anchor=east]{$1$};

\draw [->, thick] (4.5,-3.25)-- (4.5, -4.5);
\node at (4.5,-3.25)  [shape = rectangle, draw]{};
\draw (4.5,-3.25) node[anchor=east]{$x$};
\draw [red, ->, thick] (4.5,-3.25)-- (3.55, -1.6);

\draw [->, thick] (6.5,-3.25)-- (6.5, -4.5);
\node at (6.5,-3.25)  [shape = rectangle, draw]{};
\draw (6.5,-3.25) node[anchor=east]{$x^2$};

\draw [->, thick] (8.5,-3.25)-- (8.5, -4.5);
\node at (8.5,-3.25)  [shape = rectangle, draw]{};
\draw (8.5,-3.25) node[anchor=east]{$x^3$};
\draw [red, ->, thick] (8.5,-3.25)-- (7.55, -1.6);

\foreach \x in {0,1,2,3}
\foreach \y in {0, 1, 2, 3, 4}
\node at (2.5 +2*\x, -3.5 -0.25 * \y )[fill=red, inner sep=1pt, shape=circle, draw]{};

\draw [->, thick] (3.5,-0.75)-- (3.5, -2.0);
\node at (3.5,-0.75)  [shape = rectangle, draw]{};
\draw (3.5, -0.75) node[anchor=east]{$\vbh$};

\draw [->, thick] (5.5,-0.75)-- (5.5, -2.0);
\node at (5.5,-0.75)  [shape = rectangle, draw]{};
\draw (5.5,-0.75) node[anchor=east]{$\vbh x$};
\draw [red, ->, thick] (5.5,-0.75)-- (4.55, 1.0);

\draw [->, thick] (7.5,-0.75)-- (7.5, -2.0);
\node at (7.5,-0.75)  [shape = rectangle, draw]{};
\draw (7.5,-0.75) node[anchor=east]{$\vbh x^2$};

\draw [->, thick] (9.5,-0.75)-- (9.5, -2.0);
\node at (9.5,-0.75)  [shape = rectangle, draw]{};
\draw (9.5,-0.75) node[anchor=east]{$\vbh x^3$};
\draw [red, ->, thick] (9.5,-0.75)-- (8.55, 1.0);

\foreach \x in {0,1,2,3}
\foreach \y in {0, 1, 2, 3, 4}
\node at (3.5 +2*\x, -1.0 -0.25 * \y )[fill=red, inner sep=1pt, shape=circle, draw]{};

\draw [->, thick] (4.5,1.75)-- (4.5, 0.5);
\node at (4.5,1.75)  [shape = rectangle, draw]{};
\draw (4.5, 1.75) node[anchor=east]{$\vbh^2$};

\draw [->, thick] (6.5,1.75)-- (6.5, 0.5);
\node at (6.5,1.75)  [shape = rectangle, draw]{};
\draw (6.5,1.75) node[anchor=east]{$\vbh^2x$};
\draw [red, ->, thick] (6.5,1.75)-- (5.6, 3.45);

\draw [->, thick] (8.5,1.75)-- (8.5, 0.5);
\node at (8.5,1.75)  [shape = rectangle, draw]{};
\draw (8.5,1.75) node[anchor=east]{$\vbh^2x^2$};

\foreach \x in {0,1,2}
\foreach \y in {0, 1, 2, 3, 4}
\node at (4.5 +2*\x, 1.5 -0.25 * \y )[fill=red, inner sep=1pt,
shape=circle, draw]{};

\draw [->, thick] (5.5,4.25)-- (5.5, 3.0);
\node at (5.5,4.25)  [shape = rectangle, draw]{};
\draw (5.5, 4.25) node[anchor=east]{$\vbh^3$};

\draw [->, thick] (7.5,4.25)-- (7.5, 3.0);
\node at (7.5,4.25)  [shape = rectangle, draw]{};
\draw (7.5,4.25) node[anchor=east]{$\vbh^3x$};

\draw [->, thick] (9.5,4.25)-- (9.5, 3.0);
\node at (9.5,4.25)  [shape = rectangle, draw]{};
\draw (9.5,4.25) node[anchor=east]{$\vbh^3x^2$};

\foreach \x in {0,1,2}
\foreach \y in {0, 1, 2, 3, 4}
\node at (5.5 +2*\x, 4.0 -0.25 * \y )[fill=red, inner sep=1pt, shape=circle, draw]{};

\node at (0.5,-3.25)  [shape = rectangle, draw]{};
\draw (0.5,-3.25) node[anchor=north]{$x^{-1}$};
\node at (-1.5,-3.25)  [shape = rectangle, draw]{};
\draw (-1.5,-3.25) node[anchor=north]{$x^{-2}$};
\node at (-3.5,-3.25)  [shape = rectangle, draw]{};
\draw (-3.5,-3.25) node[anchor=north]{$x^{-3}$};

\draw [->, thick] (-0.5, -3.25)--(-0.5,-2.0);
\draw [red, ->, thick] (-0.5,-2.5)-- (-1.45, -0.9);
\draw [->, thick] (-2.5, -3.25)--(-2.5,-2.0);
\draw [->, thick] (-4.5, -3.25)--(-4.5,-2.0);
\draw [red, ->, thick] (-4.5,-2.5)-- (-5.45, -0.75);

\foreach \x in {0,1,2}
\foreach \y in {0, 1, 2, 3, 4 }
\node at (-0.5-2*\x, -3.25 +0.25 * \y )[fill=red, inner sep=1pt, shape=circle, draw]{};

\node at (1.5,-0.75)  [shape = rectangle, draw]{};
\draw (1.5,-0.75) node[anchor=north]{$\vbh x^{-1}$};
\node at (-0.5,-0.75)  [shape = rectangle, draw]{};
\draw (-0.5,-0.75) node[anchor=north]{$\vbh x^{-2}$};
\node at (-2.5,-0.75)  [shape = rectangle, draw]{};
\node at (-4.5,-0.75)  [shape = rectangle, draw]{};

\draw [->, thick] (0.5, -0.75)--(0.5,0.5);
\node at (0.5,-0.75) [fill=red, inner sep=1pt, shape=circle, draw]{};
\draw [red, ->, thick] (0.5,-0.0)-- (-0.45, 1.8);
\draw [->, thick] (-1.5, -0.75)--(-1.5,0.5);
\node at (-1.5,-0.75) [fill=red, inner sep=1pt, shape=circle, draw]{};
\draw [->, thick] (-3.5, -0.75)--(-3.5,0.5);
\node at (-3.5,-0.75) [fill=red, inner sep=1pt, shape=circle, draw]{};
\draw [red, ->, thick] (-3.5,-0.0)-- (-4.45, 1.8);

\foreach \x in {0,1,2}
\foreach \y in {0, 1, 2, 3, 4 }
\node at (0.5-2*\x, -0.75 +0.25 * \y )[fill=red, inner sep=1pt, shape=circle, draw]{};2
\node at (2.5,1.75)  [shape = rectangle, draw]{};
\draw (2.5,1.75) node[anchor=north]{$\vbh^2 x^{-1}$};
\node at (0.5,1.75)  [shape = rectangle, draw]{};
\node at (-1.5,1.75)  [shape = rectangle, draw]{};
\node at (-3.5,1.75)  [shape = rectangle, draw]{};

\draw [->, thick]  (1.5, 1.75)--(1.5,3.0);
\node at (1.5,1.75) [fill=red, inner sep=1pt, shape=circle, draw]{};
\draw [red, ->, thick] (1.5, 2.5)--(0.55,4.2);
\draw [->, thick]  (-0.5, 1.75)--(-0.5,3.0);
\node at (-0.5,1.75) [fill=red, inner sep=1pt, shape=circle, draw]{};
\draw [->, thick] (-2.5, 1.75)--(-2.5,3.0);
\node at (-2.5,1.75) [fill=red, inner sep=1pt, shape=circle, draw]{};
\draw [red, ->, thick] (-2.5, 2.5)--(-3.55,4.2);
\draw [->, thick] (-4.5, 1.75)--(-4.5,3.0);
\node at (-4.5,1.75) [fill=red, inner sep=1pt, shape=circle, draw]{};

\foreach \x in {0,1,2, 3}
\foreach \y in {0, 1, 2, 3, 4 }
\node at (1.5-2*\x, 1.75 +0.25 * \y )[fill=red, inner sep=1pt,
shape=circle, draw]{};

\node at (3.5,4.25)  [shape = rectangle, draw]{};
\draw (3.5,4.25) node[anchor=north]{$\vbh^3 x^{-1}$};
\node at (1.5,4.25)  [shape = rectangle, draw]{};
\node at (-0.5,4.25)  [shape = rectangle, draw]{};
\node at (-2.5,4.25)  [shape = rectangle, draw]{};
\node at (-4.5,4.25)  [shape = rectangle, draw]{};

\draw [->, thick]  (2.5, 4.25)--(2.5,5.5);
\node at (2.5,4.25) [fill=red, inner sep=1pt, shape=circle, draw]{};
\draw [->, thick]  (0.5, 4.25)--(0.5,5.5);
\draw [->, thick] (-1.5, 4.25)--(-1.5,5.5);
\node at (-1.5,4.25) [fill=red, inner sep=1pt, shape=circle, draw]{};
\draw [->, thick] (-3.5, 4.25)--(-3.5,5.5);

\foreach \x in {0,1,2, 3}
\foreach \y in {0, 1, 2, 3, 4 }
\node at (2.5-2*\x, 4.25 +0.25 * \y )[fill=red, inner sep=1pt, shape=circle, draw]{};
\end{tikzpicture}
\caption{ The $\vb$-Bockstein
   spectral sequence: $E_1=H\Zu_{\bigstar}^Q[\vbh]\Rightarrow
   \kR_{\bigstar}^Q.$ Squares are copies of $\Z$, small (red) dots are
   copies of $\Ftwo$. The $d^1$ differentials are indicated by
   displaying the top or bottom degree component between the relevant
   boxes. The larger (blue) blobs select out the degrees occurring in
   one slice spectral sequence.}
\end{figure}

\newpage
\begin{figure}
\begin{tikzpicture}
[scale =1.2]
\clip (-5, -5.5) rectangle (10, 6);
\draw[step=0.25, gray, very thin] (-9,-9) grid (15, 9);

\draw[blue, ->, line width=0.7mm](-10,-4.5)--(10, -4.5);
\draw (9.2, -4.5) node[anchor=north]{$n$};
\draw[blue, thick](-10,-2)--(10, -2);
\draw[blue, thick](-10,0.5)--(10, 0.5);
\draw[blue, thick](-10,3.0)--(10, 3.0);
\draw[blue, thick](-10,5.5)--(10, 5.5);
\draw (-3.5,-4.5) node[anchor=north]{$-6$};
\draw (-1.5,-4.5) node[anchor=north]{$-4$};
\draw (0.5,-4.5) node[anchor=north]{$-2$};
\draw (2.5,-4.5) node[anchor=north]{$0$};
\draw (4.5,-4.5) node[anchor=north]{$2$};
\draw (6.5,-4.5) node[anchor=north]{$4$};
\draw (8.5,-4.5) node[anchor=north]{$6$};

\draw[blue, ->, line width=0.7mm](2,-4.5)--(2, 6);
\draw (2,5.5) node[anchor=south east]{$s$};
\draw[blue,  thick](-5,-4.5)--(-5, 6);
\draw[blue,  thick](-4,-4.5)--(-4, 6);
\draw[blue,  thick](-3,-4.5)--(-3, 6);
\draw[blue,  thick](-2,-4.5)--(-2, 6);
\draw[blue,  thick](-1,-4.5)--(-1, 6);
\draw[blue,  thick](0,-4.5)--(0, 6);
\draw[blue,  thick](1,-4.5)--(1, 6);
\draw[blue,  thick](3,-4.5)--(3, 6);
\draw[blue,  thick](4,-4.5)--(4, 6);
\draw[blue,  thick](5,-4.5)--(5, 6);
\draw[blue,  thick](6,-4.5)--(6, 6);
\draw[blue,  thick](7,-4.5)--(7, 6);
\draw[blue,  thick](8,-4.5)--(8, 6);
\draw[blue,  thick](9,-4.5)--(9, 6);

\foreach \y in {-1}
\node at (10.5 +\y, 2.75 + 2.25 * \y ) [circle, draw=blue!50,
fill=blue!20] {}; 

\foreach \y in {-2, -1, 0, 1}
\node at (8.5 +\y, 2.25 + 2.25 * \y ) [circle, draw=blue!50,
fill=blue!20] {}; 

\foreach \y in {-2, -1, 0, 1}
\node at (6.5 +\y, 1.75 + 2.25 * \y ) [circle, draw=blue!50,
fill=blue!20] {}; 

\foreach \y in {0, 1, 2 ,3}
\node at (2.5 +\y, -3.25 + 2.25* \y ) [circle, draw=blue!50,fill=blue!20] {}; 

\foreach \y in {0, 1, 2,3}
\node at (0.5 +\y, -3.75 + 2.25 * \y ) [circle, draw=blue!50, fill=blue!20] {}; 




\draw [->, thick] (2.5,-3.25)-- (2.5, -4.5);
\node at (2.5,-3.25)  [shape = rectangle, draw]{};
\draw (2.5,-3.25) node[anchor=east]{$1$};

\node at (4.5,-3.25)  [shape = circle, draw]{};
\draw (4.5,-3.25) node[anchor=east]{$2x$};

\draw [->, thick] (6.5,-3.25)-- (6.5, -4.5);
\node at (6.5,-3.25)  [shape = rectangle, draw]{};
\draw (6.5,-3.25) node[anchor=east]{$x^2$};

\node at (8.5,-3.25)  [shape = circle, draw]{};
\draw (8.5,-3.25) node[anchor=east]{$2x^3$};

\foreach \x in {0,2}
\foreach \y in {0, 1, 2, 3, 4}
\node at (2.5 +2*\x, -3.5 -0.25 * \y )[fill=red, inner sep=1pt, shape=circle, draw]{};

\draw [->, thick] (3.5,-0.75)-- (3.5, -1.25);
\node at (3.5,-0.75)  [shape = rectangle, draw]{};
\draw (3.5, -0.75) node[anchor=east]{$\vbh$};

\node at (5.5,-0.75)  [shape = circle, draw]{};
\draw (5.5,-0.75) node[anchor=east]{$2\vbh x$};

\draw [->, thick] (7.5,-0.75)-- (7.5, -1.25);
\node at (7.5,-0.75)  [shape = rectangle, draw]{};
\draw (7.5,-0.75) node[anchor=east]{$\vbh x^2$};

\node at (9.5,-0.75)  [shape = circle, draw]{};
\draw (9.5,-0.75) node[anchor=east]{$2\vbh x^3$};

\foreach \x in {0,2}
\foreach \y in {0, 1}
\node at (3.5 +2*\x, -1.0 -0.25 * \y )[fill=red, inner sep=1pt, shape=circle, draw]{};

\draw [->, thick] (4.5,1.75)-- (4.5, 1.25);
\node at (4.5,1.75)  [shape = rectangle, draw]{};
\draw (4.5, 1.75) node[anchor=east]{$\vbh^2$};

\node at (6.5,1.75)  [shape = circle, draw]{};
\draw (6.5,1.75) node[anchor=east]{$2\vbh^2x$};

\draw [->, thick] (8.5,1.75)-- (8.5, 1.25);
\node at (8.5,1.75)  [shape = rectangle, draw]{};
\draw (8.5,1.75) node[anchor=east]{$\vbh^2x^2$};

\foreach \x in {0,2}
\foreach \y in {0, 1}
\node at (4.5 +2*\x, 1.5 -0.25 * \y )[fill=red, inner sep=1pt,
shape=circle, draw]{};

\draw [->, thick] (5.5,4.25)-- (5.5, 3.75);
\node at (5.5,4.25)  [shape = rectangle, draw]{};
\draw (5.5, 4.25) node[anchor=east]{$\vbh^3$};

\node at (7.5,4.25)  [shape = circle, draw]{};
\draw (7.5,4.25) node[anchor=east]{$2\vbh^3x$};

\draw [->, thick] (9.5,4.25)-- (9.5, 3.75);
\node at (9.5,4.25)  [shape = rectangle, draw]{};
\draw (9.5,4.25) node[anchor=east]{$\vbh^3x^2$};

\foreach \x in {0,2}
\foreach \y in {0, 1}
\node at (5.5 +2*\x, 4.0 -0.25 * \y )[fill=red, inner sep=1pt, shape=circle, draw]{};

\node at (0.5,-3.25)  [shape = rectangle, draw]{};
\draw (0.5,-3.25) node[anchor=north]{$x^{-1}$};
\node at (-1.5,-3.25)  [shape = rectangle, draw]{};
\draw (-1.5,-3.25) node[anchor=north]{$x^{-2}$};
\node at (-3.5,-3.25)  [shape = rectangle, draw]{};
\draw (-3.5,-3.25) node[anchor=north]{$x^{-3}$};

\draw [->, thick] (-0.5, -3.25)--(-0.5,-2.75);
\draw [->, thick] (-2.5, -3.25)--(-2.5,-2.0);
\draw [->, thick] (-4.5, -3.25)--(-4.5,-2.75);

\foreach \x in {0,2}
\foreach \y in {0, 1, 2 }
\node at (-0.5-2*\x, -3.25 +0.25 * \y )[fill=red, inner sep=1pt,
shape=circle, draw]{};

\foreach \y in {0, 1, 2, 3, 4, 5 }
\node at (-0.5-2, -3.25 +0.25 * \y )[fill=red, inner sep=1pt, shape=circle, draw]{};

\node at (1.5,-0.75)  [shape = rectangle, draw]{};
\draw (1.5,-0.75) node[anchor=north]{$\vbh x^{-1}$};
\node at (-0.5,-0.75)  [shape = rectangle, draw]{};
\draw (-0.5,-0.75) node[anchor=north]{$\vbh x^{-2}$};
\node at (-2.5,-0.75)  [shape = rectangle, draw]{};
\node at (-4.5,-0.75)  [shape = rectangle, draw]{};

\draw [->, thick] (0.5, -0.75)--(0.5,-0.25);
\node at (0.5,-0.75) [fill=red, inner sep=1pt, shape=circle, draw]{};
\draw [->, thick] (-3.5, -0.75)--(-3.5,-0.25);
\node at (-3.5,-0.75) [fill=red, inner sep=1pt, shape=circle, draw]{};

\foreach \x in {0,2}
\foreach \y in {0, 1, 2}
\node at (0.5-2*\x, -0.75 +0.25 * \y )[fill=red, inner sep=1pt,
shape=circle, draw]{};

\node at (2.5,1.75)  [shape = rectangle, draw]{};
\draw (2.5,1.75) node[anchor=north]{$\vbh^2 x^{-1}$};
\node at (0.5,1.75)  [shape = rectangle, draw]{};
\node at (-1.5,1.75)  [shape = rectangle, draw]{};
\node at (-3.5,1.75)  [shape = rectangle, draw]{};

\draw [->, thick]  (1.5, 1.75)--(1.5,2.25);
\node at (1.5,1.75) [fill=red, inner sep=1pt, shape=circle, draw]{};
\draw [->, thick] (-2.5, 1.75)--(-2.5,2.25);
\node at (-2.5,1.75) [fill=red, inner sep=1pt, shape=circle, draw]{};

\foreach \x in {0,2}
\foreach \y in {0, 1, 2 }
\node at (1.5-2*\x, 1.75 +0.25 * \y )[fill=red, inner sep=1pt,
shape=circle, draw]{};

\node at (3.5,4.25)  [shape = rectangle, draw]{};
\draw (3.5,4.25) node[anchor=north]{$\vbh^3 x^{-1}$};
\node at (1.5,4.25)  [shape = rectangle, draw]{};
\node at (-0.5,4.25)  [shape = rectangle, draw]{};
\node at (-2.5,4.25)  [shape = rectangle, draw]{};
\node at (-4.5,4.25)  [shape = rectangle, draw]{};

\draw [->, thick]  (2.5, 4.25)--(2.5,4.75);
\node at (2.5,4.25) [fill=red, inner sep=1pt, shape=circle, draw]{};
\draw [->, thick] (-1.5, 4.25)--(-1.5,4.75);

\foreach \x in {0,2}
\foreach \y in {0, 1, 2}
\node at (2.5-2*\x, 4.25 +0.25 * \y )[fill=red, inner sep=1pt, shape=circle, draw]{};
\end{tikzpicture}

\caption{$E_2=E_{\infty}$ of the $\vb$-Bockstein
   spectral sequence: $E_1= H\Zu_{\bigstar}^Q[\vbh]\Rightarrow
   \kR_{\bigstar}^Q. $ Squares are copies of $\Z$. Large circles are
   copies of $\Z$ of index 2 in $E_1$.  Small (red) dots are
   copies of $\Ftwo$. The larger (blue) blobs select out the degrees occurring in
   one slice spectral sequence. The Gap occurs in this picture most
   prominently as the three zeroes below the diagonal squares emanating
 from $(n,s)=(-2,0)$, but it also requires three zeroes along
 the diagonal emanating from $(n, 0)$ for any $n\leq -1$. }
\end{figure}



\subsection{The slice spectral sequence (SSS)}
\label{subsec:slice}
The SSS is usually presented as a mechanism for calculating the integer
graded part $\pi^Q_*(X)$, with $RO(Q)$-graded part obtained by 
looking at suspensions of $X$ \cite{HillSlice}.  Since suspension by $\rho$ just shifts
the slices, we may combine the suspensions and view it as a method for 
calculating the $RO(Q)$-graded homotopy. We will not give a detailed
account of the SSS here since in this special case it is contained in
the BSS. This short section is designed
so that  those already familiar with the SSS can pick it out from the
BSS as described in the previous subsection.  

For $\kR$, the slice filtration is precisely the same as
the $\vb$-Bockstein filtration, so that the $RO(Q)$-graded BSS and SSS
are the same. The one piece of added value in the SSS is the way that
the usual ($\Z$-graded) display highlights vanishing lines and `The
Gap'.  This information is present in the BSS displays above, but it
is not so easy to spot. 

Just to be completely explicit, for $\kR$ the negative slices are all
zero and the odd slices are all zero. The positive even slices are
suspensions of $H\Zu$: in the notation of \cite{HHR}
$P^{2k}_{2k}k\R=\Sigma^{k\rho} H\Zu$ for $k\geq 0$.  In  Adams grading
$(n,s)=(2k-s, s)$ the usual SSS has
$$\pi^Q_{2k-s}(P^{2k}_{2k}k\R)
=H_{2k-s}^Q(S^{k\rho};\Zu) =H_{k-s}^Q(S^{k\sigma};\Zu)
=H\Zu^Q_{k-s-k\sigma} , $$
so that the homotopy groups of one particular slice occur along a line of
slope $-1$. The appropriate chart is displayed clearly as Figure 7 in
\cite[Page 414]{HHRCfour}. 

Note that the homotopy group above, at coordinates $(x_S,y_S)=(2k-s,s)$ in the
SSS, appears at coordinates $(x_B,y_B)=(2k-s ,k)$ in the BSS.
Thus $(x_B,y_B)=(x_S,(x_S+y_S)/2)$ and $(x_S,y_S)=(x_B, 2y_B-x_B)$.  
The BSS differential $d^1_B$ corresponds to the SSS differential
$d^3_S$. The slice spectral  sequence simply selects a subspectral sequence given by 
one particular $\sigma$-grading in each box.

The larger (pale blue) blobs in the two   $\vb$-{\bf B}SS charts are the
potentially nonzero
entries in the {\bf S}SS for $\pi^Q_*(\kR)$ (i.e., the integer graded
part).  We have displayed five diagonal rows of blobs (in adding
$(1,1)$ to $(n,s)$, $\sigma$ is subtracted). The diagonal
emanating from $(n,s)=(-2,0)$ does not contain any non-zero SSS entries. 
The diagonal emanating from $(n,s)=(0,0)$ contains $1$, and $a^k\vbh^k$  
(corresponding to $\eta^k$) for $k\geq 1$. Of course $1, a\vbh$ and $a^2\vbh^2$
survive to $E^{\infty}$ to give $\eta$ and $\eta^2$ and higher powers are the images of
differentials.
The first non-zero entry on the diagonal emanating from $(n,s)=(2,0)$ is on
the  $s=2$ line, where it is $\vbh^2x$, and after that $a^k\vbh^{2+k}x$
for $k\geq 1$. These
all support nonzero differentials, so only $2\vbh^2x$ survives (to
give the generator of $\pi^Q_4(\kR)=\pi_4(ko)$). In our display, only one blob is
visible on diagonal emanating from $(n,s)=(4,0)$, and it is zero.


\begin{thebibliography}{999}
\bibitem{Araki} S.Araki. 
``Orientations in $\tau$-cohomology theories''
 Japan. J. Math. (N.S.), {\bf 5}(2) (1979) 403-430
\bibitem{Atiyah}
M. F. Atiyah 
``$K$-theory and reality''
 Quart. J. Math. Oxford Ser. (2), {\bf 17} (1966) 367–386
\bibitem{root} 
R.R.Bruner and J.P.C.Greenlees 
``The algebraic Bredon-L\"offler conjecture.” Experimental Mathematics 4 (1995) 289-297.
\bibitem{kobg} 
R.R.Bruner and J.P.C.Greenlees 
``The real connective $K$-theory of finite groups.''
AMS Surveys and Monographs (2010) 318+v pp
\bibitem{Carlsson}
G. Carlsson,
``Equivariant stable homotopy and Segal's Burnside ring conjecture.''
 Ann. of Math. (2) 120 (1984), no. 2, 189–224.
\bibitem{Caruso} J.L. Caruso
``Operations in equivariant $\Z/p$-cohomology''
 Math. Proc. Cambridge Philos. Soc. 126 (1999), no. 3, 521–541.
\bibitem{Dress}
A.W.M.Dress, 
``Contributions to the theory of induced representations.''
 Algebraic K-theory, II: "Classical'' algebraic K-theory and connections with arithmetic (Proc. Conf., Battelle Memorial Inst., Seattle, Wash., 1972), pp. 183-240. Lecture Notes in Math., Vol. 342, Springer, Berlin, 1973. 
\bibitem{Dugger} D. Dugger
``An Atiyah-Hirzebruch spectral sequence for $KR$-theory'' 
$K$-theory, {\bf 35} (2005) 213-256
\bibitem{assiet} J.P.C.Greenlees 
 ``Adams spectral sequences in equivariant topology.'' 
  (Cambridge University Thesis, 1985)
\bibitem{GTate}
J.P.C.Greenlees, 
``Representing Tate cohomology of G-spaces.''
 Proc. Edinburgh Math. Soc. (2) 30 (1987), no. 3, 435–443. 
\bibitem{cc}  
J.P.C.Greenlees 
  ``Stable maps into free G-spaces'', 
  Transactions of the American Mathematical Society, 
  {\bf 310} (1988), 199-215.
\bibitem{Tate} J.P.C.Greenlees and J.P. May
  ``Generalized Tate cohomology.''
   Mem. Amer. Math. Soc. 113 (1995), no. 543, viii+178 pp. 
\bibitem{BPRn}  
J.P.C.Greenlees and L.Meier
 ``Gorenstein duality for real spectra.''
AGT (to appear)  60pp, arXiv: 1607.02332
\bibitem{HillSlice}
M.A.Hill 
``The equivariant slice filtration: a primer''
 Homology Homotopy Appl., {\bf 14} (2) (2102) 143– 166, 2012
\bibitem{HHR} M. Hill, M.J.Hopkins and D.C.Ravenel 
``On the nonexistence of elements of Kervaire invariant one''
Ann. Math {\bf 184} (2016), 1-262
\bibitem{HHRCfour} M. Hill, M.J.Hopkins and D.C.Ravenel 
``The slice spectral sequence for the $C_4$ analog of real $K$-theory'
Forum Math. {\bf 29} (2017) 383-447
\bibitem{HM}
M.A.Hill and L. Meier
``The $C_2$-spectrum $TMF_1(3)$ and its invertible modules.''
AGT (to appear). 
\bibitem{HK}
P.Hu and I. Kriz 
``Real-oriented homotopy theory and an analogue of the Adams-Novikov
spectral sequence''
Topology, 40(2):317–399, 2001. 
\bibitem{L}
P.S. Landweber
``Conjugations on complex manifolds and equivariant homotopy of $MU$''
 Bull. Amer. Math. Soc., 74:271–274, 1968. 13
\bibitem{lmsm} L.G.Lewis, J.P.May and M.Steinberger (with contributions
by J.E.McClure) ``Equivariant stable homotopy theory''
Lecture notes in mathematics {\bf 1213}, Springer-Verlag (1986) 
\bibitem{Ricka} N. Ricka
``Equivariant Anderson duality and Mackey functor duality''
Glasgow Math. J. (2014) 1-28
\bibitem{Stong1} R.Stong Letter to J.P.May, dated 4 September 1980. 
\bibitem{Stong2} R.Stong 
``Ordinary cohomology of points'' 
1980 13pp. 
\bibitem{Waner} S.Waner
``Notes on the cohomology of points'' 
1980. 
\end{thebibliography}
\end{document}